\documentclass{article}
\usepackage[english]{babel}
\usepackage{amsmath, amsxtra, latexsym,amscd, amsthm, amssymb,indentfirst}
\usepackage[mathscr]{eucal}
\usepackage{nccmath}
\usepackage{enumerate}
\usepackage{hyperref}
\usepackage{textcomp}
\usepackage{array, tabularx, longtable}%
\usepackage{multicol}
\usepackage{color}

\def\R{\mathbb R}
\def\N{\mathbb N}

\newcommand*{\mr}{\rm{mr}}
\def\E{\mathbb E}

\def\ds{\displaystyle\sum}

\def\rank{{\rm rank}}
\def\sr{{\rm sr}}

\def\G{\mathcal{G}}
\def\I{\mathcal I}

\def\int{{\rm int\,}}

\def\B{\mathcal B}
\def\A{\mathcal A}

\def\C{\mathcal C}
\def\M{\mathcal M}
\def\oA{\overline{A}}

\def\det{{\rm det }}
\def\sgn {{\rm sgn}}
\def\eps{\varepsilon}

\def\rank{{\rm rank}}

\theoremstyle{plain}
\newtheorem{theorem}{Theorem}[section]
\theoremstyle{theorem}
\numberwithin{theorem}{section}

\newtheorem{lemma}[theorem]{Lemma}
\newtheorem{proposition}[theorem]{Proposition}
\theoremstyle{definition}
\newtheorem{definition}[theorem]{Definition}

\newtheorem{example}[theorem]{Example}

\newtheorem{notation}[theorem]{Notation}

\begin{document}
\title{An algebra for the propagation of errors in matrix calculus}
\date{}

\author{Nam Van Tran\\ Faculty of Applied Sciences\\
	HCMC University of Technology and Education, Vietnam\\
	namtv@hcmute.edu.vn\\
	\\
	Imme van den Berg\\ Department of Mathematics, University of \'{E}vora,  Portugal\\ ivdb@uevora.pt}
\maketitle
\begin{abstract}
We assume that every element of a matrix has a small, individual error, and model it by an external number, which is the sum of a nonstandard real number and a neutrix, the latter being a convex (external) set having the group property. The algebra of external numbers formalizes common error analysis, with rules for calculation which are a sort of mellowed form of the axioms for real numbers.

We extend the algebra of external numbers to matrix calculus. Many classical properties continue to hold, sometimes stated in terms of inclusion instead of equality. There are notable exceptions, for which we give counterexamples and investigate suitable adaptations. In particular we study addition and multiplication of matrices, determinants, near inverses, and generalized notions of linear independence and rank.

\textbf{Keywords}: matrix calculus, error propagation, independence, rank, external numbers.
	
\textbf{AMS classification}: 03H05, 15A03, 15A09, 15B33, 65F99.
\end{abstract}
%\tableofcontents

\section{Introduction}

In many mathematical models input data and output data are given in the form of vectors and matrices. The data often are imprecise, which may have several origins, for example imperfect knowledge, measurement problems, changes in time, model reductions, rounding off etc. and the imprecisions  can have various sizes. In this article the imprecisions are modelled by (scalar) \emph{neutrices}, which are convex subgroups of the set of nonstandard numbers, most of them are external sets. Then each entry of a matrix is an \emph{external number}, which is the pointwise (Minkowski) sum of a (nonstandard) real number and a neutrix. Every entry has its own individual neutrix, modelling the diversity of imprecisions. The neutrix usually is infinitely small with respect to the real number, and then the external number is called \emph{zeroless}. Examples of neutrices are the external set of infinitesimals $ \oslash $ and the external set $ \pounds  $ of numbers smaller in absolute value than some standard real number, as well as all multiples of them, but there exist other types of neutrices \cite{Koudjeti Van den Berg}. The term neutrix is borrowed from Van der Corput, and we were inspired by his \emph{Ars Neglegendi} \cite{Van der Corput}.

Within the setting of external numbers we study the effects of error propagation in calculations with matrices and determinants. 

We recognize the properties which are supposed to hold for operations in error analysis in the definition of addition and multiplication below, given in terms of the Minkowski operations. We consider the particular case of zeroless external numbers  $\alpha=a+A, \beta=b+B$, where $ a,b $ are real numbers and $A,B$ are two neutrices. Then we have (see also Definition \ref{defnam})

%\begin{enumerate}
	%\item \label{congI} $\alpha+\beta=a+ b+A+B$. 	
	%\item  $-\alpha=-a+A$.
	%\item \label{nhanI}  $\alpha \beta=ab+Ab+Ba$.
	%\item \label{1/alphaI} $\dfrac{1}{\alpha}=\dfrac{1}{a}+\dfrac{A}{a^2}$.	
%\end{enumerate}
\begin{equation}\label{erroranalysis}
\begin{aligned}
	 \alpha+\beta&=a+ b+A+B \\	
	 -\alpha&=-a+A\\
	 	   \alpha \beta&=ab+Ab+Ba\\
	 \dfrac{1}{\alpha}&=\dfrac{1}{a}+\dfrac{A}{a^2}.	
\end{aligned}
\end{equation}
Classical error analysis is more or less informal, for instance the above rules correspond to "provisional rules" for propagation of errors of \cite{Taylor}, to hold approximately and somewhat adhoc, using common sense. In contrast, in terms of external numbers, the equalities of \eqref{erroranalysis} are part of formal mathematics and permit us to prove much more general laws, which lead to the notion of complete arithmetical solid in \cite{dinisberg 2016}. Addition and multiplication satisfy the properties of a completely regular commutative semigroup \cite{Petrich}, and adapted forms of distributivity, order relation, Dedekind completeness and the Archimedean property are shown to hold.

We cannot hope for such strong rules for matrix calculus, still the matrices form a regular commutative semigroup for addition: the usual laws for addition are valid, but the sum of a matrix and its "inverse element for addition" will be a matrix of neutrices, and not the zero-matrix. Also in many cases the common laws for multiplication of matrices do hold.  Problems may appear when multiplying matrices with entries of different sign, in particular when some entries are almost equal in absolute value but opposite, or when the matrix has a small determinant. Still general conditions can be given for algebraic properties to hold, sometimes in the form of inclusions instead of equalities; typically entries should not be \emph{nearly opposite}, a notion defined in Section \ref{neutrices and external numbers}.

We pay special attention to invertibility, linear dependence and independence, and rank.
  
The product of two matrices with non-zero neutrices will usually be a matrix with non-zero neutrices, so it is to be expected, that analogously to the case of addition, generically we will never obtain the identity matrix. We speak of a \emph{near inverse} if we obtain the identity matrix up to neutrices included in $\oslash$. We give conditions for near inverses to exist, in particular in the form of the adjugate matrix. In order to avoid the blow up of neutrices, the determinant should not be too small, though sometimes it may be infinitesimal.

We give a straightforward definition for linear independence of external numbers, and relate it to classical linear independence and dependence of vectors of representatives, i.e vectors of real numbers included in the external vectors.

There are several ways to define the rank of matrices of external numbers, a rank $ r $ coming directly from independence or a \emph{minor rank}, defined by the nonsingularity of minors. In fact a mixed notion \emph{strict rank} happens to be the more operational. We give conditions for its existence, and show that then the rank $ r $ is equal to the minor rank.

The present approach to propagation of errors in matrix calculus is preceded by other studies of matrices and vectors with external numbers, starting with \cite{Immedecompositionofneutrices}, where it was shown that a neutrix in standard dimension $ n $ was the direct sum of $ n $ scalar neutrices, in orthogonal directions. In \cite{Jus} and \cite{Jus2}, determinants and \emph{flexible systems}, i.e systems of linear equations such that the coefficient matrix and constant term contain external numbers, are studied in the case the former is nonsingular. Singular systems are studied in \cite{Nam} and \cite{Nam2}, which contain also an approach to linear programming under uncertainties along these lines.

%{\color{red} In \cite{dinisnam imme} an calculus for external numbers as well as its applications were presented. }

Our approach to treatment of error respects various features of common error analysis. There the errors around a value are represented by a small interval, and  an external number is a convex set, in the form of a real number plus a neutrix; also if the approximated value is non-zero, the size of the neutrix is infinitely small with respect to this value. Furthermore the interval related to the set of errors is somewhat arbitrary, often it is an upper bound and is susceptible to allow for some modification. This arbitrariness of size is reflected by the group property of  neutrices. This makes that the external number satisfies the \emph{Sorites} property: it is invariant under at least some shifts, and also some additions and multiplications. We stress the point that our approach is not functional, which is not always natural for uncertainties; however parameters are allowed, and in  principle the occurrence of a multiple of parameters does not significantly augment the complexity of the calculus.

To our opinion the approach by external numbers respects the uncertainty of error sets, while allowing for a calculus for error propagation exceeding other approaches, as attested by the axioms of a complete arithmetical solid of \cite{dinisberg 2016}.

Many of the approaches are functional, and though perhaps linear operations are preserved, they do not respond very well to multiplication, and do not allow for a proper order relation. The $o(.), O(.)-$notation of classical asymptotics may be interpreted by sets of functions \cite{Bruijn}, allowing for some calculus, but functions may be oscillatory, and as such cannot be ordered. The approach has been generalized by Van der Corput's neutrix calculus. His functional neutrices are groups of functions without unity, but again order is not respected. Functional dependence and absence of order is also a drawback for the Infinit\"{a}rcalc\"{u}l of Du Bois-Reymond, and, although ordered, Hardy-fields suppose conditions of smoothness \cite{Hardy} which are not always natural in the context of error analysis.

Other classical approaches to uncertainties include statistical and stochastical methods \cite{Clifford}, \cite{Sentivity}, \cite{Kall}, fuzzy mathematics \cite{Wierman}, \cite{Kall}, multi-parametric methods, where uncertainties depend on parameters taking values in specific domains \cite{GalNedoma}, \cite{Gollmer}, and methods based on interval calculus in various settings \cite{Neumaier}, \cite{Alefeld}, \cite{Moore}, \cite{Gabrel}. They have their own modelling purposes, random influences in the case of statistics and stochastics, properties which hold partially for the fuzzy approach, while the multi-parametric approach deals with individual treatment of errors and interval calculus manages lower and upper bounds for errors. 

Except for the latter, all are functional methods, which thus are maybe not efficient for advanced error propagation. The set-theoretic calculus of intervals has better properties when applying operations, still, due to problems of subdistributivity and intersection not in all cases simple laws can be given, and algebraic operations need not to respect order. 

Compared to the statistical, stochastical or fuzzy approach, the approach by external numbers, being deterministic with obvious membership relation, is closer to the multi-parametric approach and interval calculus. With both methods it shares the property of individual treatment of errors. However, the well-defined boundaries of intervals and parameter domains harm a strong calculus of error-propagation, which in our approach is overcome by the flexibility of the external numbers, which due to the Sorites property absorb at least some shifts, additions and multiplications. Also our approach to error-propagation starts directly from the straightforward formalization of the rules of error analysis given in \eqref{erroranalysis}. 

The above mentioned classical approaches are of courser easier to implement, still we defend that the calculus of external numbers, vectors and matrices yields insights at a intermediate level between qualitative and quantitative analysis, resulting from direct calculations of moderate complexity.

This article has the following structure. We start by recalling some properties of neutrices and external numbers, in Section \ref{neutrices and external numbers}. In Section \ref{section1c3} we show that almost all common properties of operations on matrices hold for non-negative matrices, and give general conditions for these properties to hold beyond. Section \ref{section2c3} deals with the  determinant and its minors. In section \ref{inverse matrix} we study invertible matrices. In classical algebra a matrix is invertible if and only if it is non-singular and the inverse matrix is represented through the adjugate matrix. In this context, it may not hold, but we will present conditions guaranteeing that a non-singular matrix is (nearly) invertible, with its near-inverse matrix still represented by the adjugate matrix. 
In Section \ref{section3c3} we extend the notions of linear dependence and independence to external vectors. The relationships between linear dependence and linear independence of a set of external vectors as well between external vectors and their representatives  are investigated. In section \ref{section4c3} we study the rank of a matrix with external numbers.  In classical linear algebra the rank of a matrix determined via determinants is equal to the maximum number of independent row vectors, but in  our context this relation is less evident. Different notions of rank are given, called row-rank, minor-rank, and strict rank, the latter taking into account a matrix of representatives. The minor rank is less than or equal to the row rank. However under some conditions we have equality, in particular if the strict rank is well-defined. 

%\section{Preliminaries}\label{preliminaries}

\section{Neutrices and external numbers}\label{neutrices and external numbers}

Neutrices and external numbers are well-defined external sets in the axiomatic $HST$ for nonstandard analysis as given by Kanovei and Reeken in \cite{KanoveiReeken}. This is an extension of a bounded form of Nelson's Internal Set Theory $ IST $ \cite{Nelson}. This theory extends common set theory $ ZFC $ by adding an undefined predicate "standard" to the language of set theory, and three new axioms. Introductions to $ IST $ are contained in e.g. \cite{Dienerreeb}, \cite{Dienernsaip} or \cite{Lyantsekudryk}. An important feature is that infinite sets always have nonstandard elements. In particular nonstandard numbers are already present within $ \R $. \emph{Limited} numbers are real numbers bounded in absolute value by standard natural numbers. Real numbers larger in absolute value than limited numbers are called \emph{unlimited}. Its reciprocals, together with $ 0 $, are called \emph{infinitesimal}. Limited numbers which are not infinitesimal are called \emph{appreciable}.

A {\em (scalar) neutrix} is an additive convex subgroup of the set of nonstandard real numbers $\R$.  Except for $\{0\}$ and $\R$, all neutrices are external sets.   The set of all limited numbers $\pounds$ and the set of all infinitesimals  $\oslash$  are neutrices. Note that $\pounds$ and $\oslash$ are not sets in the sense of $ ZFC $, for they are bounded subsets of $ \R $ without lowest upper bound. 
Let $\varepsilon\in \R$ be a positive infinitesimal. Some other neutrices are  $\varepsilon \oslash, \ \varepsilon\pounds$,  $\displaystyle\bigcap_{st(n)\in \N}[-\varepsilon^n, \varepsilon^n]=\pounds\varepsilon^{\not\infty}$, $\displaystyle\bigcup_{st(n)\in \N}[-e^{-1/(n\eps)}, e^{-1/(n\eps)}]=\pounds e^{-@/\eps}$; here $@  $ denotes the external set of positive appreciable numbers and $ \not\hskip -0.16cm \infty $
the external set of positive unlimited numbers.

An {\em external number}  is the Minkowski-sum of a real number and a neutrix. 
So each external number  has the form $\alpha=a+A=\{a+x|x\in A\}$, where $A$ is called the {\em neutrix part} of $\alpha$, denoted by $N(\alpha)$, and $a\in \R$ is called a {\em representative}  of $\alpha$. If $ N(\alpha)=\{0\} $, we may identify $ \{a\} $ and $ a $. If $0\not\in \alpha=a+N(\alpha)$, we call $\alpha$ {\em zeroless}. 

The collection of all neutrices is not an external set, but a definable class, denoted by $\mathcal{N}$. Also the external  numbers form a class, denoted by $ \E $.

Addition, subtraction, multiplication and division are given by the Minkowski operations below. We list also some elementary properties. For more details on neutrices and external numbers we refer to \cite{ Koudjeti Van den Berg, Immedecompositionofneutrices, Dinis, dinisberg 2016}.

Let $\alpha=a+A, \beta=b+B$ be two external numbers and $A, B$ be two neutrices.
\begin{definition}\label{defnam}
	\begin{enumerate}
		\item \label{cong} $\alpha\pm\beta=a\pm b+A+B$.		
		\item \label{nhan}  $\alpha \beta=ab+Ab+Ba+AB=ab+\max\{aB, bA, AB\}.$
		\item \label{1/alpha} If $\alpha$ is zeroless, $\dfrac{1}{\alpha}=\dfrac{1}{a}+\dfrac{A}{a^2}.$		
		\end{enumerate}
\end{definition}

If $\alpha$ and $ \beta$ are zeroless, in Definition \ref{defnam}.\ref{nhan} we may neglect the neutrix product $ AB $. For division of neutrices $A, B\in \mathcal{N}$ we use the common notation for division of groups
$$ A:B=\{c\in \R\ |\ cB\subseteq A\}.$$
Neutrices are obviously ordered by inclusion. An order relation for all external numbers $ \alpha,\beta $ is given by
\begin{equation*}
\alpha\leq\beta\equiv\forall a\in \alpha \exists b\in \beta(a\leq b).
\end{equation*}
An external number $\alpha$ is called \emph{positive} if $0<x$ for all $x\in \alpha$ and \emph{negative} if $x<0$ for all $x\in \alpha$. The number $\alpha$ is \emph{non-negative} if $0\leq \alpha$, i.e. if there exists $x\in \alpha$ such that $ 0\leq x$ and \emph{non-positive} if $ 0\geq \alpha$; this means there exists $x\in \alpha$ with $0\geq x$. Note that a neutrix is both non-negative and non-positive. The order relation is shown to be compatible with the operations, with some small adaptations \cite{Koudjeti Van den Berg}. 
\begin{proposition}\label{propnam}
	\begin{enumerate}
		\item  \label{tcnt1}$A+B=\max\{A, B\}$. 
		\item \label{tcnt2}  $\pounds A=A.$ 
		\item \label{giao} If $\alpha$ is zeroless, then $\alpha\cap \oslash \alpha=\emptyset.$		
		\item \label{tcoptiv} If $\beta$ is zeroless and $ N $ is a neutrix one has $N\beta=bN$ and $\dfrac{N}{\beta}=\dfrac{N}{b}$.
		\item \label{beta/alpha} If $\beta$ is zeroless, we have $\dfrac{\alpha}{\beta}=\dfrac{\alpha \beta}{b^{2}}=\dfrac{a}{b}+\dfrac{1}{b^2}\max\{aB, bA\}.$
	\end{enumerate}
\end{proposition}

The usual properties for the algebraic operations are preserved, like associativity and commutativity. Although the distributivity law does not hold, we always obtain subdistributivity as shown in  Part \ref{tcopti1} of Theorem \ref{tcopti}  below; Part \ref{tcopti3} characterizes the validity of distributivity with the help of the notions of \emph{relative uncertainty} and \emph{oppositeness}. Both the definition and the theorem below are from \cite{Dinis}, which contains illustrative examples and proofs.
\begin{definition}
	Let $\alpha=a+A$ and $  \beta=b+B $ be external numbers and $ C $ be a neutrix.
	\begin{enumerate}
		\item The \emph{relative uncertainty} $ R(\alpha) $ of $\alpha$  is defined by $ A/\alpha $ if $\alpha $ is zeroless, otherwise $ R(\alpha)=\R $.
		\item $\alpha$ and $  \beta$ are \emph{opposite} with respect to $ C $ if $ (\alpha+\beta)C\subset \max (\alpha C, \beta C). $
	\end{enumerate}
\end{definition}

\begin{lemma}\label{relAa}Let $\alpha=a+A$ be zeroless. Then $R(\alpha)=A/a\subseteq \oslash.$	
\end{lemma}

\begin{lemma}\label{relativeprecision} 
	Let $n\in \N$ be standard and $\alpha_{1},\dots, \alpha_n$ be external numbers. Let $\lambda=\alpha_1\cdots \alpha_n$. Then $R(\lambda)=\max\limits_{1\leq i\leq n}R(\alpha_i)$.
\end{lemma}
\begin{proof} Let $\alpha_i=a_i+A_i.$ If $ \max\limits_{1\leq i\leq n}R(\alpha_i)=\R$, there exists $i_0\in \{1,\dots, n\}$ such that $\alpha_{i_0}$ is a neutrix. It follows that $\lambda$ is a neutrix. Hence $R(\lambda)=\R=\max\limits_{1\leq i\leq n}R(\alpha_i).$ Otherwise, 
\begin{alignat*}{2}\lambda=& a_1\dots a_n +\ds_{\substack{p=1}}^n \sum\limits_{1\leq i_1<\cdots<i_p\leq n} \big(A_{i_1}\dots A_{i_p} \Pi_{j\in J} a_j\big)\\
= & a_1\dots a_n +A_1a_2\dots a_n+\cdots+A_n a_1\dots a_{n-1}\\
+&\ds_{\substack{p=2}}^n \sum\limits_{1\leq i_1<\cdots<i_p\leq n} \big(A_{i_1}\dots A_{i_p} \Pi_{j\in J} a_j\big),
\end{alignat*}  where $J=\{1,\dots, n\}\setminus \{i_1,\dots, i_p\}.$

For each $p\in \{2,\dots, n\}$ and $i_1,\dots, i_p\in \{1,\dots,n\}$ put $\mu_p=a_1\cdots a_n+ \sum\limits_{1\leq i_1<\cdots<i_p\leq n}\big(A_{i_1}\dots A_{i_p} \Pi_{j\in J} a_j\big)$. Then $$R(\lambda)=\ds_{i=1}^n R(\alpha_i)+\sum_{p=2}^n  R(\mu_p)=\max\limits_{1\leq i\leq n}R(\alpha_i)+\sum_{p=2}^n  R(\mu_p).$$ Because $R(\alpha_i)\subseteq \oslash$, we derive that $R(\alpha_{i_1})\cdots R(\alpha_{i_p})\leq \max\limits_{1\leq i\leq n} R(\alpha_i)$ for all $p\in \{2,\dots, n\}$ and $i_1,\dots, i_p\in \{1,\dots, n\}$. As a result $R(\mu_p)\leq \max\limits_{ 1\leq i\leq n} R(\alpha_i).$ So $\ds_{p=2}^n   R(\mu_p)\leq \ds_{p=2}^n   \max\limits_{1\leq i\leq n} R(\alpha_i)=\max\limits_{1\leq i\leq n} R(\alpha_i).$ Hence  $R(\lambda)=\max\limits_{1\leq i\leq n}R(\alpha_i)$.
\end{proof}

\begin{theorem}\label{tcopti}Let $\alpha=a+A, \beta=b+B, \gamma=c+C$ be external numbers. Then
	\begin{enumerate}
		\item \text{\rm (Subdistributivity)} \label{tcopti1} $(\alpha+\beta)\gamma\subseteq\alpha \gamma+\beta \gamma$.  
		\item \text{\rm (Distributivity with correction term)} $\alpha \gamma+\beta \gamma=(\alpha+\beta)\gamma+C \alpha+C \beta$.   \label{tcopti2}
		\item \text{\rm (Criterium for distributivity)}\label{tcopti3} $\alpha \gamma+\beta \gamma=(\alpha+\beta)\gamma$ if and only if    
		$R(\gamma)\subseteq \max( R(\alpha), R(\beta))$, or  $\alpha $ and  $\beta $ are not opposite with respect to $ C $. 
	\end{enumerate}	
\end{theorem}
\begin{proposition}\label{sumexternalnumber}
	Let $n\in \N$ be standard, $\alpha, \beta_1,\dots, \beta_n$ be external numbers. If $R(\alpha)\leq \min\limits_{1\leq i\leq n} R(\beta_i)$, we have $\alpha(\beta_1+\cdots+\beta_n)=\alpha\beta_1+\cdots+\alpha\beta_n$.
\end{proposition}
\begin{proof}
	It is proved by using External Induction \cite{Nelson} and Theorem \ref{tcopti}.\ref{tcopti3}.
\end{proof}
Obviously distributivity holds if $ \alpha $ and $ \beta $ are of the same sign, say if we are always working with positive numbers or non-negative numbers. The following generalization may have some practical value. 

\begin{definition}\label{defnearop}
Two zeroless elements $ \alpha, \beta \in \E$ are \emph{nearly opposite} if $\alpha/\beta\subseteq  -1+\oslash$.
\end{definition} 
 For example, a real number $ b\simeq 1 $ and $ -1 $ are nearly opposite, but two different standard real numbers are not nearly opposite.
\begin{proposition}\label{notnearlyopposite}
	Let $ \alpha,\beta, \gamma \in \E$ such that $ \alpha$  and $ \beta $ are not nearly opposite. Then $ (\alpha+\beta)\gamma=\alpha\gamma +\beta\gamma $.
\end{proposition}
\begin{proof}
If $ \gamma=c+C $ is zeroless, $ R(\gamma)\subseteq \oslash $, hence $ \alpha$  and $ \beta $ are not  opposite with respect to $R(\gamma)  $. Hence, applying Theorem~\ref{tcopti}.\ref{tcopti3}, $$(\alpha+\beta)\gamma=(\alpha+\beta)c(1+R(\gamma))=\alpha c(1+R(\gamma))+\beta c(1+R(\gamma))=\alpha\gamma +\beta\gamma. $$
Assume now that $ \gamma=C $. We have always distributivity if $ \alpha $ or $ \beta $ is neutricial. In case both are zeroless, we may suppose that $ |\alpha|\leq |\beta| $. Then with $ \beta=b+B $ we have $ \left|\frac{\alpha}{b}\right| \leq 1+\oslash$ and $ 1+\frac{\alpha}{b}\subset@ $, hence also $ 1+R(\beta)+\frac{\alpha}{b}\subset@  $, so by Lemma~\ref{relAa} and Theorem~\ref{tcopti}.\ref{tcopti3}
$$ (\alpha+\beta)C=b(1+R(\beta)+\frac{\alpha}{b}) C\subseteq b@C=\beta C=\max(\alpha C,\beta C)=\alpha C+\beta C.$$
We conclude that distributivity holds in each case.
\end{proof}
\begin{definition}\label{neutrix part 2} 
	Let $N$ be a neutrix and $\alpha$ be an external number. The external number $\alpha$ is called an \emph{absorber} of  $N$ if $\alpha N\subset N,$ and an \emph{exploder} of $N$ if $N\subset \alpha N.$ 
\end{definition}
We have $tA=A$ for all $|t|\in @$, so appreciables are neither absorbers nor exploders. Infinitesimals are absorbers of $ \pounds $ and $ \oslash $, and unlimited numbers are exploders of these neutrices. Observe that if  $\varepsilon\in \R$ is a positive infinitesimal, it is not an absorber of  $\pounds\varepsilon^{\not\infty}$, nor of $\pounds e^{-@/\eps}$, and its reciprocal $ 1/\varepsilon $ is not an exploder for these neutrices.

\section{Matrices with external numbers}\label{section1c3}
In this section operations on matrices with external numbers are studied. We start with addition, which satisfies the rules of a regular commutative semigroup. Then we turn to the algebra of scalar multiplication respectively matrix multiplication, which may give rise to inclusions instead of equalities, in particular as regards associativity and distributivity. We present conditions to guarantee equalities.

We will consider matrices of the form 

$$\A=\begin{pmatrix}
\alpha_{11} & \alpha_{12}& \cdots & \alpha_{1n}\\
\vdots & \vdots&  \ddots & \vdots \\
\alpha_{m1} & \alpha_{m2} & \cdots & \alpha_{mn}
\end{pmatrix},$$
where $m,n\in \N$ and $\alpha_{ij}\in \E$ for $ 1\leq i\leq m, 1\leq j\leq n$; the natural numbers $ m,n $	are always supposed to be standard. We use the common notation $\A=(\alpha_{ij})_{m\times n}$. The \emph{transpose} of the matrix $\A$ is defined by $ \A^T=(\nu_{ij})_{n\times m}$ with $\nu_{ij}=\alpha_{ji}$ for all $i=1,\dots, n; j=1,\dots, m$. 
 
 A matrix $\A=(\alpha_{ij})_{m\times n}$ is called \emph{neutricial} if all elements of $\A$ are neutrices, and \emph{zeroless} if all of its entries are zeroless.
 If $ \alpha_{ij}= a_{ij}+A_{ij}$ for all $ 1\leq i\leq m, 1\leq j\leq n$, the matrix $ (a_{ij})_{m\times n} $ is called a \emph{matrix of representatives} and $ N(\A):=(A_{ij})_{m\times n} $ the \emph{associated neutricial matrix}. We denote by $\M_{m,n}(F)$ the set of all $m\times n$ matrices over $F$, where $F$ is either $\R$ or $\E$. When $m=n$ we simply write $\M_n(F)$.

For matrices $\A=(\alpha_{ij})_{n\times n}\equiv (a_{ij}+A_{ij})_{n\times n}\in \M_{n}(\E)$ we define 

$$\overline{A}=\max\limits_{\substack{1\leq i, j\leq n}} A_{ij}, \underline{A}=\min\limits_{\substack{ 1\leq i, j\leq n}} A_{ij}, |\overline{\alpha}|=\max\limits_{\substack{1\leq i,j\leq n}} \left| \alpha_{ij}\right|, |\underline{\alpha}|=\min\limits_{\substack{ 1\leq i, j\leq n}} |\alpha_{ij}|.$$
Operations on  $\M_{m, n}(\E)$ are defined similarly as in classical linear algebra. Let $\A=(\alpha_{ij})_{m\times n}\in \M_{m,n}(\E), \B=(\beta_{ij})_{m\times n}\in \M_{m,n}(\E), \C=(\gamma_{ij})_{n\times p} \in \M_{n,p}(\E) $ and $\lambda\in \E$. Then

$$\A+\B=(\alpha_{ij}+\beta_{ij})_{m\times n}$$
$$ \lambda\A=(\lambda\alpha_{ij})_{m\times n}$$
$$\A\C=(\mu_{ij})_{m\times p}$$ with $\mu_{ij}=\ds_{k=1}^n \alpha_{ik}\gamma_{kj}$ for all $i=1,\dots, m, j=1,\dots, p.$ 

The associative law and commutative law for addition hold for external numbers, hence also for matrices. This makes $\M_{m, n}(\E)$ a commutative semigroup for addition. Let $\A\in\M_{m, n}(\E)$. Then $ \A+(N(\A)) =\A$, and $ \A+(-\A) =N(\A)$. Hence $ \A+(-\A+\A) =\A$, so the commutative semigroup $\M_{m, n}(\E)$ is regular. If also $\mathcal{O}\in\M_{m, n}(\E)$ is a neutrix and $ \A+\mathcal{O}=A $, then $ N(\A)+\mathcal{O} =N(\A)$; note that $ N(\A) $ is in a sense a maximal individualized neutral element, because $ \mathcal{O}_{ij}\subseteq N(\A)_{ij} $ for all $i=1,\dots, n; j=1,\dots, m$. Next proposition resumes the above observations.

\begin{proposition}\label{algmat}
Let $\A, \B, \C\in \M_{m, n}(\E)$. Let $\mathcal{O}\in \M_{m, n}(\E)$ be neutricial. Then 
\begin{enumerate}
	\item $\A+(\B+\C)=(\A+\B)+C$.
	\item $\A+\B=\B+\A$
	\item $\A+ \mathcal{O}=\A$ if and only if $\mathcal{O}_{ij}\subseteq (N(\A))_{ij}$ for all $i=1,\dots, n, j=1,\dots, m$.
	\item $\A + (-\A)=N(\A)$.	
\end{enumerate}
As a consequence $ \M_{m, n}(\E) $ is a commutative regular semigroup for addition.
\end{proposition}
We denote the zero matrix of arbitrary dimension by $O$. Observe that $\M_{m, n}(\E)$ is also a monoid for addition since $\A+O=O+\A=\A$ for all $\A \in \M_{m, n}(\E)$, i.e. the matrix $O$ acts a neutral element. But except for matrices with real elements we do not have $\A+-\A=O$.

In the remaining part of this section we study multiplication and its interaction with addition. We will see that almost all usual properties hold for non-negative matrices and non-negative scalars, and outside these classes they still hold under quite general conditions.
 
The following properties of scalar multiplication, multiplication by the identity matrix and transposition are proved as in classical linear algebra. 
\begin{proposition}\label{algmat2} Let $\A\in \M_{m,n}(\E)$. Then 
	\begin{enumerate}
		\item $0\A=O$. 
		\item $1\A =\A$.			
		\item $\alpha(\beta\A)=(\alpha\beta)\A$.		
	\end{enumerate}
\end{proposition}

\begin{proposition} Let $\A\in \M_{m,p}(\E), \B\in \M_{p,q}(\E)$ and $I_n$ be the identity matrix of order $n$. Then 
	\begin{enumerate}
		\item $I_m\A=\A= \A I_p$.
		\item $(\A\B)^T=\B^T\A^T$.
	\end{enumerate}

\end{proposition}
It follows from the fact that the multiplication on external numbers is not distributive that scalar multiplication and the multiplication of matrices is not distributive. The theorem below presents conditions such that the distributive property does hold.  
\begin{definition}
	Let $\A=(\alpha_{ij})_{m\times n}, \B=(\beta_{ij})_{m\times n}\in \M_{m,n}(\E)$. The matrices $\A$ and $\B$ are said to be \emph{not nearly opposite} if $\alpha_{ij},\beta_{ij}$ are not nearly opposite for any $i\in\{1,\dots, m\}, j\in \{1,\dots, n\}$. 
\end{definition}
Note that matrices with entries of the same sign, and in particular non-negative matrices are not nearly opposite. The next theorem follows readily from the criterion for distributivity given  by Theorem~\ref{tcopti}.\ref{tcopti3}, Proposition \ref{sumexternalnumber} and Proposition~\ref{notnearlyopposite}.

\begin{theorem}\label{distributivitymatrix} Let $\A=(\alpha_{ij})_{m\times n}\in \M_{m,n}(\E)$ and $ \B=(\beta_{ij})_{n\times p}, \C=(\gamma_{ij})_{n\times p}\in \M_{n,p}(\E)$. Let $\alpha,\beta\in \E$.
	
	\begin{enumerate}
		\item \label{congmatrix1}	If either $R(\alpha)\leq \min\limits_{\substack{1\leq i\leq m\\1\leq j\leq n}}\max\left\{R(\alpha_{ij}), R(\beta_{ij})\right\}$ or $\A,\B$ are not nearly opposite,  then $$\alpha(\A+\B) =\alpha\A+\alpha\B.$$
		\item \label{nhanmatrix1}	If either $\max\limits_{\substack{1\leq i\leq m\\1\leq j\leq n}}\left\{R(\alpha_{ij}) \right\}\leq \max \{R(\alpha), R(\beta)\}$ or $\alpha,\beta$ are not nearly opposite, then $(\alpha+\beta)\A = \alpha\A+\beta\A$.
		
		\item \label{nhanmatrix2} If either $\max\limits_{\substack{1\leq i\leq m\\1\leq j\leq n}} R(\alpha_{ij})\leq \min\limits_{\substack{1\leq i\leq m\\1\leq j\leq n}}\max \{R(\beta_{ij}), R(\gamma_{ij})\}$ or $\B,\C$ are not nearly opposite, then  $\A(\B+\C)= \A\B+\A\C$.
	\end{enumerate} 
\end{theorem}

\begin{proof}
	Properties \ref{congmatrix1}. and \ref{nhanmatrix1}. follow directly from Theorem \ref{tcopti}.\ref{tcopti3}.
		
	As for Part \ref{nhanmatrix2}, let $\A(\B+\C)=(\mu_{ij})_{m\times n}$, $\A\B=(\lambda_{ij})_{m\times p}, \A\C=(\nu_{ij})_{m\times p}$. If either $\max\limits_{\substack{1\leq i\leq m\\1\leq j\leq n}} R(\alpha_{ij})\leq \min\limits_{\substack{1\leq i\leq m\\1\leq j\leq n}}\max \{R(\beta_{ij}), R(\gamma_{ij})\}$ or $\B, \C$ are not nearly opposite, by Theorem \ref{tcopti}.\ref{tcopti3} and Proposition~\ref{notnearlyopposite} $\alpha_{ij}(\beta_{rs}+\gamma_{rs})=\alpha_{ij}\beta_{rs}+\alpha_{ij}\gamma_{rs}$ for all $1\leq i\leq m, 1\leq j, r\leq n, 1\leq s\leq p.$ As a result, 
	\begin{alignat*}{2}
	\mu_{ij}=&\alpha_{i1}(\beta_{1j}+\gamma_{1j})+\cdots+\alpha_{in}(\beta_{nj}+\gamma_{nj}
	)\\
	=&\big(\alpha_{i1}\beta_{1j}+\cdots+\alpha_{in}\beta_{nj}\big)+\big( \alpha_{i1}\gamma_{1j}+\cdots+\alpha_{in}\gamma_{nj}\big)\\
	=& \lambda_{ij}+\nu_{ij}.
	\end{alignat*} 
\end{proof}

 The next proposition gives conditions for distributivity in the case of  zeroless matrices, in terms of minimal or maximal relative uncertainty. 
\begin{proposition}
	Let $\A=(\alpha_{ij})_{n\times n}= (a_{ij}+A_{ij})_{n\times n}\in \M_{n}(\E), \B=(\beta_{ij})_{n\times n}= (b_{ij}+B_{ij})_{n\times n}\in \M_{n}(\E), \C=(\gamma_{ij})_{n\times n}= (c_{ij}+C_{ij})_{n\times n}\in \M_{n}(\E)$ be zeroless matrices. Let $\alpha,\beta\in \E$.
	\begin{enumerate}
		\item \label{nhan0} If $R(\alpha)\leq \max \{\underline{B}/\overline{\beta}, \underline{C}/\overline{\gamma})\}$, then $$\alpha(\A+\B) =\alpha\A+\alpha\B.$$
		\item \label{nhan11}	If $\dfrac{\overline{A}}{\underline{\alpha}}\leq \max \{R(\alpha), R(\beta)\}$, then $(\alpha+\beta)\A = \alpha\A+\beta\A.$
		\item \label{nhan12}	If $\dfrac{\overline{A}}{\underline{\alpha}}\leq \max \{\underline{B}/\overline{\beta}, \underline{C}/\overline{\gamma})\}$, then $\A(\B+\C)= \A\B+\A\C$.
	\end{enumerate}
\end{proposition}
\begin{proof} 
	\begin{enumerate}
	\item 	For all $1\leq i\leq m, 1\leq j\leq n$ it holds that
	 $$\max \{\underline{B}/\overline{\beta}, \underline{C}/\overline{\gamma})\}\leq \max \{R(\beta_{ij}), R(\gamma_{ij})\}. $$
	 Then distributivity follows from Part \ref{congmatrix1} of Theorem \ref{distributivitymatrix}. 
	
	\item  The distributivity  follows from the fact that $\max\limits_{\substack{1\leq i\leq m\\1\leq j\leq n}} R(\alpha_{ij})\leq \dfrac{\overline{A}}{\underline{\alpha}}$ and Part \ref{nhanmatrix1} of Theorem \ref{distributivitymatrix}. 
	\item  The distributivity is a consequence of Part \ref{nhan11}, the fact that 
	$$\max \{\underline{B}/\overline{\beta}, \underline{C}/\overline{\gamma})\}\leq \max \{R(\beta_{ij}), R(\gamma_{ij}) $$
	 for all $1\leq i\leq m, 1\leq j\leq n$, and Part \ref{nhanmatrix2} of Theorem \ref{distributivitymatrix}. 
	
	\end{enumerate}

\end{proof}

The subdistributivity property for external numbers implies the following general properties of subdistributivity for scalar multiplication and multiplication of matrices. The proofs are immediate.

\begin{proposition} Let $\A=(\alpha_{ij})_{m\times n}, \B=(\beta_{ij})_{m\times n}\in \M_{m, n}(\E), \C=(\gamma_{ij})_{n\times p}\in \M_{n, p}(\E)$. Let $\alpha,\beta\in \E$.Then 
	\begin{enumerate}
		\item $\alpha(\A+\B)\subseteq \alpha\A+\alpha\B$.
		\item $(\alpha+\beta)\A \subseteq \alpha\A+\beta\A$.
		\item $A(\B+\C)\subseteq \A\B+\A\C$.
		\item $(\A+\B)\C\subseteq \A\C+\B\C$.
	\end{enumerate}
\end{proposition}

The lack of distributivity also implies that the multiplication of matrices is not associative \cite[p.35]{Jus2}. For example, let $A=\begin{pmatrix}
1&1\\
0&0
\end{pmatrix}, B=\begin{pmatrix}
1&0\\
-1&0
\end{pmatrix}, C=\begin{pmatrix}
\oslash\\
\oslash
\end{pmatrix}$. One has
$$(AB)C=\begin{pmatrix}
\begin{pmatrix}
1&1\\
0&0
\end{pmatrix}\begin{pmatrix}
1&0\\
-1&0
\end{pmatrix}
\end{pmatrix}\begin{pmatrix}
\oslash\\
\oslash
\end{pmatrix}=\begin{pmatrix}
0\\0
\end{pmatrix}$$
and 
$$A(BC)=\begin{pmatrix}
1&1\\
0&0
\end{pmatrix}\begin{pmatrix}
\begin{pmatrix}
1&0\\
-1&0
\end{pmatrix}\begin{pmatrix}
\oslash \\ \oslash 
\end{pmatrix}
\end{pmatrix}=\begin{pmatrix}
1&1\\
0&0
\end{pmatrix}\begin{pmatrix}
\oslash\\\oslash
\end{pmatrix}=\begin{pmatrix}
\oslash\\
0
\end{pmatrix}.$$
So $(AB)C\not=A(BC).$
However, the subdistributivity of multiplication of external numbers, as shown in Property \eqref{tcopti1} of Proposition \ref{tcopti}, implies the following properties of inclusion.
\begin{proposition}
	Let $\A=(\alpha_{ij})_{m\times n}\in \M_{m, n}(\E), \B=(\beta_{ij})_{n\times p}\in \M_{n, p}(\E)$ and $ \C=(\gamma_{ij})_{p\times q}\in \M_{p, q}(\E)$. Then
	\begin{enumerate}
		\item \label{inclusion1}  $(\A\B)\C\subseteq \A(\B\C) \mbox{ if $\A$ is a real matrix or $\B, \C$ are non-negative}.$
	\item  \label{inclusion2}  $\A(\B\C)\subseteq (\A\B)\C \mbox{ if $\C$ is a real matrix or $\A, \B$ are non-negative}.$
\end{enumerate}
\end{proposition}
\begin{proof} Let $\A\B\equiv \mathcal{D}\equiv(\delta_{ij})_{m\times p}$, $ \B\C\equiv\mathcal{E}\equiv (\varepsilon_{ij})_{n\times q} $, $(\A\B)\C\equiv (\eta_{ij})_{m\times q}$ and $\A(\B\C)\equiv (\theta_{ij})_{m\times q}$. 
	
\ref{inclusion1}. We have by subdistributivity for all $ i\in \{1, \dots, m\}, k\in \{1,\dots, q\} $

\begin{equation*}
	\eta_{ik}=\sum_{j=1}^p\delta_{ij}\gamma_{jk}=\sum_{j=1}^p\left(\sum_{r=1}^n \alpha_{ir}\beta_{rj}\right) \gamma_{jk}\subseteq \sum_{j=1}^p\sum_{r=1}^n \alpha_{ir}\beta_{rj} \gamma_{jk}.
	\end{equation*}
	If  $\A$ is a real matrix, or else by non-negative of $\B \C$, the last sum is equal to $ \displaystyle\sum_{r=1}^n\alpha_{ir} \left(\sum_{j=1}^p \beta_{rj}\gamma_{jk}\right) =\theta_{ik}$.

	\ref{inclusion2}. The proof is similar to the proof of Part \ref{inclusion1}. 
	
\end{proof}
We below provide conditions for the associative law  for the multiplication of matrices to be valid.

\begin{proposition}\label{associate of product of matrix2}
	Let $\A=(\alpha_{ij})_{m\times n}\in \M_{m,n}(\E)$, $\B=(B_{ij})_{n\times p}\in \M_{n,p}(\E)$ be a neutricial matrix and $\C=(\gamma_{ij})_{p\times q}\in \M_{p, q}(\E)$. Then $\A(\B\C)=(\A\B)\C.$	
\end{proposition}
\begin{proof}
	Put $\A\B\equiv \mathcal{D}\equiv(\delta_{ij})_{m\times p}.$ One has $$\delta_{ij} =\alpha_{i1}B_{1j}+\cdots+\alpha_{in}B_{nj}$$ for all $1\leq i\leq m, 1\leq j\leq p.$ Because $B_{ij}$ is a neutrix for  $1\leq i\leq n,q\leq j\leq p$, also $\delta_{ij}$ is a neutrix for all $1\leq i\leq m, 1\leq j\leq p.$ 
	
	So $(\A\B)\C=\mathcal{D}\C\equiv (\eta_{ij})_{m\times q}$, where \begin{alignat}{2}\label{comutative of matrix2} \eta_{ij}=&\delta_{i1}\gamma_{1j}+\cdots+\delta_{ip}\gamma_{pj}\notag\\
	=& (a_{i1}B_{11}+\cdots+a_{in}B_{n1})\gamma_{1j} +\cdots+ (a_{i1}B_{1p}+\cdots+a_{in}B_{np})\gamma_{pj}.
	\end{alignat} 
	On the other hand, let $\B\C\equiv \mathcal{E}\equiv (\varepsilon_{ij})_{n\times q}$. Then 
	$\varepsilon_{ij}=B_{i1}\gamma_{1j}+\cdots+B_{ip}\gamma_{pj}$ for all $1\leq i\leq n, 1\leq j\leq q.$ Put $\A(\B\C)=\A\mathcal{E}\equiv (\theta_{ij})_{m\times q}$. Using the distributivity property for multiplication by neutrices we have for all $1\leq i\leq m, 1\leq j\leq q,$
	\begin{alignat}{2}\label{comutative22} 
	\theta_{ij}=& \alpha_{i1}\varepsilon_{1j}+\cdots+\alpha_{in}\varepsilon_{nj}\notag\\
	=& \alpha_{i1}(B_{11}\gamma_{1j}+\cdots+B_{1p}\gamma_{pj})+\cdots+ \alpha_{in}(B_{n1}\gamma_{1j}+\cdots+B_{np}\gamma_{pj})\notag\\
	=& \alpha_{i1}B_{11}\gamma_{1j}+\cdots+\alpha_{i1}B_{1p}\gamma_{pj}+\cdots+ \alpha_{in}B_{n1}\gamma_{1j}+\cdots+\alpha_{in}B_{np}\gamma_{pj}\notag\\
	=&(\alpha_{i1}B_{11}+\cdots+\alpha_{in}B_{n1})\gamma_{1j} +\cdots+ (\alpha_{i1}B_{1p}+\cdots+\alpha_{in}B_{np})\gamma_{pj}.
	\end{alignat} 
	From \eqref{comutative of matrix2} and \eqref{comutative22} one has $\eta_{ij}=\theta_{ij}$ for all $1\leq i\leq m, 1\leq j\leq q.$
	Hence $(\A\B)\C=\A(\B\C)$.
\end{proof}

\begin{proposition}\label{associate of product of matrix}
	Let $\A=(a_{ij})_{m\times n}\in \M_{m,n}(\R)$, $\B=(\beta_{ij})_{n\times p}\in \M_{n,p}(\E)$, and $\C=(c_{ij})_{p\times q}\in \M_{p, q}(\R)$. Then $\A(\B\C)=(\A\B)\C.$	
\end{proposition}
\begin{proof}
	The proof is as above, now using the distributivity property for multiplication by real numbers. Put $\A\B\equiv \mathcal{D}\equiv(\delta_{ij})_{m\times p}.$ One has $$\delta_{ij} =a_{i1}\beta_{1j}+\cdots+a_{in}\beta_{nj}$$ for all $1\leq i\leq m, 1\leq j\leq p.$

	As a consequence, $(\A\B)\C=\mathcal{D}C\equiv(\eta_{ij})_{m\times q}$, where \begin{alignat}{2}\label{comutative of matrix} \eta_{ij}=&\delta_{i1}c_{1j}+\cdots+\delta_{ip}c_{pj}\notag\\
	=& (a_{i1}\beta_{11}+\cdots+a_{in}\beta_{n1})c_{1j} +\cdots+ (a_{i1}\beta_{1p}+\cdots+a_{in}\beta_{np})c_{pj}.
	\end{alignat} 
	On the other hand, let $\B\C\equiv \mathcal{E}\equiv (\varepsilon_{ij})_{n\times q}$. Then 
	$\varepsilon_{ij}=\beta_{i1}c_{1j}+\cdots+\beta_{ip}c_{pj}$ for all $1\leq i\leq n, 1\leq j\leq q.$ Put $\A(\B\C)=\A\mathcal{E}\equiv (\theta_{ij})_{m\times q}$. Then for all $1\leq i\leq m, 1\leq j\leq q,$ 
	\begin{alignat}{2}\label{comutative2} 
	\theta_{ij}=& a_{i1}\varepsilon_{1j}+\cdots+a_{in}\varepsilon_{nj}\notag\\
	=& a_{i1}(\beta_{11}c_{1j}+\cdots+\beta_{1p}c_{pj})+\cdots+ a_{in}(\beta_{n1}c_{1j}+\cdots+\beta_{np}c_{pj})\notag\\
	=& a_{i1}\beta_{11}c_{1j}+\cdots+a_{i1}\beta_{1p}c_{pj}+\cdots+ a_{in}\beta_{n1}c_{1j}+\cdots+a_{in}\beta_{np}c_{pj}\notag\\
	=&(a_{i1}\beta_{11}+\cdots+a_{in}\beta_{n1})c_{1j} +\cdots+ (a_{i1}\beta_{1p}+\cdots+a_{in}\beta_{np})c_{pj}.
	\end{alignat} 
	By \eqref{comutative of matrix} and \eqref{comutative2} one has $\eta_{ij}=\theta_{ij}$ for all $1\leq i\leq m, 1\leq j\leq q.$
	Hence $(\A\B)\C=\A(\B\C)$.
\end{proof}

\begin{proposition}\label{nonnegass}
	Let $\A=(\alpha_{ij})_{m\times n}\in \M_{m,n}(\E), \B=(\beta_{ij})_{n\times p}\in \M_{n,p}(\E), \C=(\gamma_{ij})_{p\times q}\in \M_{p,q}(\E)$ be non-negative matrices. Then $(\A\B)\C=\A(\B\C)$.
\end{proposition}
\begin{proof} Let $(\A\B)\C\equiv (\eta_{ij})_{m\times q}$ and $\A(\B\C)\equiv (\theta_{ij})_{m\times q}$. 
	Because the elements of the matrices $\A,\B$ and $\C$ are  non-negative, we always have distributivity. Hence for all $1\leq i\leq m, 1\leq j\leq q$ 
	$$\eta_{ij}=\ds_{k=1}^n \alpha_{ik} \ds_{r=1}^p \beta_{kr}\gamma_{rj}=\ds_{k=1}^n\ds_{r=1}^p \alpha_{ik}\beta_{kr}\gamma_{rj}=\ds_{r=1}^p\gamma_{rj} \ds_{k=1}^n \alpha_{ik}\beta_{kr} =\theta_{ij}.$$
	It follows that $(\A\B)\C=\A(\B\C)$.
\end{proof} 
Obviously, the above associative property continues to hold if the entries of each matrix all have the same sign.

From above results we see that the set of non-negative matrices together addition and scalar multiplication satisfies almost every axiom of a vector space, except for the existence of inverse elements for addition. Also distributivity and associativity of multiplication are respected. We state this observation in the theorem below. 
\begin{theorem} Let $\M_{m\times n}^+(\E)$ be the set of non-negative matrices over $\E$. Then for all $ \A,\B,\C\in \M_{m\times n}^+(\E)$ and non-negative $ \lambda,\mu\in \E $
	\begin{enumerate}
		\item $ \A+\B\in \M_{m\times n}^+(\E) $.
		\item $ \A+(\B+\C)=(\A+\B) +\C$.
		\item $O \in \M_{m\times n}^+(\E)$ and $ \A+O=O $.
		\item $ \A+\B=\B+\A$.
		\item $\lambda\A\in\M_{m\times n}^+(\E)$.% if $0\leq \lambda$.
		\item $ \lambda(\mu \A) =(\lambda\mu)\A$.
		\item $ 1\A=\A $.
		\item $ \lambda(\A+\B)=\lambda \A +\lambda \B$.
		\item $ (\lambda +\mu)\A=\lambda \A +\mu\A$.% if $0\leq \lambda, 0\leq \mu$.
	\end{enumerate}
	Moreover, whenever the product of non-negative matrices over $ \E $ is well-defined, it is distributive and associative. 
\end{theorem} 
\begin{proof}
	The theorem follows from Proposition \ref{algmat}, \ref{algmat2}, Theorem \ref{distributivitymatrix} and Propostion \ref{nonnegass}. 
\end{proof}

\section{Determinants% of matrices with external numbers
}	\label{section2c3} 

We define determinants of matrices with external numbers in the usual way through sums of signed products of entries. We show that common techniques for calculation need to be applied with care, since they often use distributivity, and thus may reduce or augment the neutrix part. As often problems arise from the absence of distributivity in the presence of (nearly) opposite terms in combination with precision. In the context of calculation of determinants and also the solution of systems the occurrence of such nearly opposite terms is natural, for one searches for zeros, though in the case of external numbers this generally results in neutrices. Information on the order of magnitude of minors and neutrix parts is useful here. In the final part we give a condition implying that determinants of triangular matrices, with a triangle of neutrices instead of zeros, still equals the product of elements on the diagonal.
\begin{definition}\cite{Jus} Let $n\in \N$ be standard. Let $\A=(\alpha_{ij})$ be an $n\times n$ matrix over $\E$. The determinant of $\A$ is the external number defined by
\begin{equation}\label{defdet}
\det(\A)=\sum\limits_{\sigma\in S_n} \sgn(\sigma) \alpha_{1\sigma(1)}\alpha_{2\sigma(2)}\dots \alpha_{n\sigma(n)}
\end{equation}
	where $S_n$ is the set of all permutations of $\{1,\dots, n\}$. We also denote the determinant of the matrix $\A$ by $|\A|$, and often by $\Delta$. Then $\Delta_{i,j}$  the $(i,j)$ is the minor of $\Delta$, that is the determinant of $(n-1)\times (n-1)$ submatrix of $\Delta$ that results from removing the $i^{th}$ row and the $j^{th}$ column of $\Delta$.
\end{definition}

The following properties are obvious and proved using similar arguments as in classical algebra.
\begin{enumerate}[(i)]
	\item $\det(\A)=\det(\A^T)$, where $A^T$ is the transposition matrix of $\A$.
	\item If we interchange two rows (columns) of a matrix, the determinant changes its sign. 
	\item The determinant of matrix which has a row of neutrices is a neutrix. 
	\item The determinant of matrix which has two identical rows (columns) is a neutrix. 
\end{enumerate}

It is tempting to see the sum of products of external numbers of \eqref{defdet} as the set of sums of products of representatives, but this is not true in general for determinants of matrices $ n\times n $ when $ n>2 $. 

Clearly, if $ \A=(\alpha )$, with $ \alpha\in \E $, then $ \det(\alpha) =\{a |a\in\alpha \}$. For $ n=2 $, let $\A=\begin{pmatrix}
\alpha_{11}& \alpha_{12}\\
\alpha_{21} &\alpha_{22}
\end{pmatrix}$. Then  
\begin{equation}\label{detrep}
\det(\A)=\alpha_{11}\alpha_{22}-\alpha_{21}\alpha_{12}=\{a_{11}a_{22}-a_{21}a_{12}| a_{ij}\in \alpha_{ij}, 1\leq i, j\leq 2\},
\end{equation}	
being Minkowski sum of Minkowski products.
Now let $$\A=\begin{pmatrix}
\alpha_{11}& \alpha_{12}& \alpha_{13}\\
\alpha_{21} &\alpha_{22}& \alpha_{23}\\
\alpha_{31} &\alpha_{32}& \alpha_{33}\\
\end{pmatrix}.$$
Then it does not always hold that the sum of products \eqref{defdet} is equal to the sum of products of representatives; in particular this means the value given by the Rule of Sarrus does not need to correspond to the set of values given by the Rule of Sarrus applied to representatives. In fact, doing so, we do not apply the Minkovski rules properly, for we choose repeatedly the same representatives. 
We give here an example.

\begin{example}\label{Sarrus} Let
	$$\A=\begin{pmatrix}
	1+\oslash & 0&0\\
	0 &1&1+\varepsilon\\
	0& 1& 1
	\end{pmatrix}$$ with $\varepsilon \simeq 0,\varepsilon \neq 0$
	Then  $ \det(\A) =\oslash$, but the set, say $ S $, of values of the Rule of Sarrus applied to the representatives satisfies $ S=-(1+\oslash)\varepsilon$.
\end{example}

Because of subdistributivity, the Laplace expansion of a determinant along a column or a row may not be equal to the determinant. For example, if we expand the determinant in Example~\ref{Sarrus} along the first column  we obtain that
$$(1+\oslash) \det\begin{pmatrix}
1&1+\eps\\
1&1
\end{pmatrix}-0 \det\begin{pmatrix}
0& 0\\
1&1
\end{pmatrix}+0 \det\begin{pmatrix}
0& 0\\
1&1+\eps
\end{pmatrix}=-(1+\oslash)\eps\subset  \oslash.$$
So using products of representatives or the Laplace expansion possibly reduces the neutrix part, and even may turn a neutricial determinant into a zeroless value. We come back to this subject when we discuss singular and non-singular matrices in Section~\ref{inverse matrix}.

In general the Laplace expansion of a determinant along a column (row) is always included in the determinant. 

\begin{proposition}[\cite{Jus2}]\label{lei}
	Let $n\in \N$ be standard. Let $\A=(\alpha_{ij})_{n\times n}\in \M_{n}(\E)$ and $\Delta=\det(\A)$. Then for all $j\in \{1,...,n\}$, 
	$$(-1)^{j+1}\alpha_{1j}\Delta_{1,j}+\cdots+(-1)^{j+n}\alpha_{nj}\Delta_{n,j}\subseteq \Delta.$$
\end{proposition}
However, if we expand along a column (row) such that the relative uncertainty of all elements in this column are less than  or equal to those of all the remaining elements,  the equality for Laplace expansion holds.  
\begin{theorem}\label{equality of determinant}
	Let $\A=(\alpha_{ij})_{n\times n}\in \M_n(\E)$. If there exists $k\in \{1,\dots, n\}$ such that \begin{equation}\label{assumption1} \max\limits_{1\leq i\leq n} R(\alpha_{ik})\leq \min\limits_{\substack{j\not=k\\ 1\leq i,j\leq n}}R(\alpha_{ij})
	\end{equation} then $$(-1)^{k+1}\alpha_{1k}\Delta_{1,k}+\cdots+(-1)^{k+n}\alpha_{nk}\Delta_{n,k}= \Delta.$$
	
\end{theorem}

\begin{proof} Without loss of generality, we assume that $k=1$. The Laplace expansion along column $k$ yields
	\begin{alignat}{2}\label{laplaceexpansion} 
	& \alpha_{11}\Delta_{1,1} -\alpha_{21}\Delta_{2,1}+\cdots +\alpha_{n1}(-1)^{1+n}\Delta_{n,1} \notag\\
	=&\alpha_{11}\displaystyle\sum_{\substack{\sigma\in S_n \\ \sigma(1)=1}} \big( \sgn(\sigma) \alpha_{\sigma(2)2}\cdots \alpha_{\sigma(n)n} \big) + \cdots + \notag\\
	&+  \alpha_{n1}\displaystyle\sum_{\substack{\sigma\in S_n \\ \sigma(1)=n}} \big(  \sgn(\sigma) \alpha_{\sigma(2)2}\cdots \alpha_{\sigma(n)n}\big).
	\end{alignat}	
	Put $\beta^\sigma_{i1}= \sgn(\sigma) \alpha_{\sigma(2)2}\cdots \alpha_{\sigma(n)n} $ with $\sigma\in S_{n}, \sigma(1)=i.$ We will show that $$\alpha_{i1}\displaystyle\sum_{\substack{\sigma\in S_n \\ \sigma(1)=i}}   \sgn(\sigma) \alpha_{\sigma(2)2}\cdots \alpha_{\sigma(n)n})=\displaystyle\sum_{\substack{\sigma\in S_n \\ \sigma(1)=i}}   \sgn(\sigma) \alpha_{i1} \alpha_{\sigma(2)2}\cdots \alpha_{\sigma(n)n}$$ for all $i\in \{1,\dots, n\}$.
	
	By Lemma \ref{relativeprecision} and assumption \eqref{assumption1}, 
	\begin{alignat*}{2}
	R(\alpha_{i1})\leq &\max\limits_{1\leq i\leq n} R(\alpha_{i1}) \leq  \min\limits_{\substack{1\leq r\leq n\\ 2\leq s\leq n}} R(\alpha_{rs})\\
\leq  & \min\limits_{\substack{1\leq r\leq n\\ 2\leq s\leq n\\r\not=i}} R(\alpha_{rs})\leq  \max\limits_{\substack{1\leq r\leq n\\ 2\leq s\leq n\\r\not=i}} R(\alpha_{rs}) =	R(\beta^\sigma_{i1}).
	\end{alignat*} 
	
	By Proposition \ref{sumexternalnumber}, it follows that $$\alpha_{i1}\displaystyle\sum_{\substack{\sigma\in S_n \\ \sigma(1)=i}}   \sgn(\sigma) \alpha_{\sigma(2)2}\cdots \alpha_{\sigma(n)n}=\displaystyle\sum_{\substack{\sigma\in S_n \\ \sigma(1)=i}}   \sgn(\sigma) \alpha_{i1} \alpha_{\sigma(2)2}\cdots \alpha_{\sigma(n)n}$$ for all $i\in \{1,\dots, n\}$.
	
	From \eqref{laplaceexpansion} one derives 
	\begin{alignat*}{2}%\label{laplaceexpansion1} 
	& \alpha_{11}\Delta_{1,1} -\alpha_{21}\Delta_{2,1}+\cdots +\alpha_{n1}(-1)^{1+n}\Delta_{n,1}\notag\\
	=& \left(\displaystyle\sum_{\substack{\sigma\in S_n \\ \sigma(1)=1}}   \sgn(\sigma) \alpha_{11} \alpha_{\sigma(2)2}\cdots \alpha_{\sigma(n)n}\right)  +\cdots+ \notag \\
	+& \left(\displaystyle\sum_{\substack{\sigma\in S_n \\ \sigma(1)=n}}   \sgn(\sigma) \alpha_{n1} \alpha_{\sigma(2)2}\cdots \alpha_{\sigma(n)n}\right) \\
	=&\displaystyle\sum_{\sigma\in S_n }  \sgn(\sigma) \alpha_{\sigma(1)1} \alpha_{\sigma(2)2}\cdots \alpha_{\sigma(n)n}=\det(\A).
	\end{alignat*} 
\end{proof}

We give now conditions for the validity of the property of addition and of multiplication by a scalar. 

The addition property $ \det(C)=\det(A)+\det(B) $ when $ B $ is equal to $ A $, except for one line, and $ C $ is obtained from $ A $ and $ B $ just summing with respect to this line does not hold in full generality. Indeed, let
$\A=\begin{pmatrix}
1&1 \\
1+\oslash&1+\oslash
\end{pmatrix}$, $\B=\begin{pmatrix}
-1&-1 \\
1+\oslash&1+\oslash
\end{pmatrix}$ and $\C=\begin{pmatrix}
0&0 \\
1+\oslash&1+\oslash
\end{pmatrix}$. Then $\det(\A)=\det(\B)=\oslash $, while $\det(\C)=0\neq \oslash=\det(\A)+\det(\B)$. General conditions for the addition property to hold are stated in the next proposition.
\begin{proposition}\label{plus of determinant} 
	Let $\B=(\beta_{ij})_{n\times n},\mathcal{C}=(\gamma_{ij})_{n\times n}\in \M_n(\E)$ be matrices which possibly differ at row $ k $, i.e. 
	$$\beta_{ij}=\begin{cases} \alpha_{ij} &\mbox{ if } i\not =k, j\in \{1,\dots, n\}\\
	\beta_{kj} &\mbox{ if } i=k, j\in \{1,\dots, n\}
	\end{cases} $$ $$\gamma_{ij}=\begin{cases} \alpha_{ij} &\mbox{ if } i\not= k, j\in \{1,\dots, n\}\\
	\gamma_{kj} &\mbox{ if } i=k, j\in \{1,\dots, n\},
	\end{cases}$$
	where all $ \alpha_{ij},\beta_{kj},\gamma_{kj}\in \E $. Let $\A=(\alpha_{ij})_{n\times n}\in \M_n(\E)$ be defined by
	$$\A=\begin{cases} \alpha_{ij} &\mbox{ if } i\not =k, j\in \{1,\dots, n\}\\
	\alpha_{kj}=\beta_{kj}+\gamma_{kj} &\mbox{ if } i=k, j\in \{1,\dots, n\}.
	\end{cases}  $$
	Then $$\det{\A}\subseteq \det(\B)+\det(\mathcal{C}).$$
	\\
	Moreover, if 
	\begin{equation}\label{dieukiendinhthuc} \max\limits_{\substack{1\leq i,j\leq n\\i\not=k}} R(\alpha_{ij})\leq \max\Big\{ \min\limits_{1\leq j\leq n} R(\beta_{kj}), \min\limits_{1\leq j\leq n} R(\gamma_{kj})\Big\}, 
	\end{equation} or $\beta_{kj}$ and $\gamma_{kj}$ are not nearly opposite for all $1\leq j\leq n$, then  $$\det(\A) = \det(\B)+\det(\mathcal{C}).$$
\end{proposition}
\begin{proof} By subdistributivity, we have 
	\begin{alignat*}{2} \det(\A)=&\ds_{\sigma\in S_n}\sgn(\sigma) \alpha_{1\sigma(1)}\cdots \alpha_{n\sigma(n)}\\
	=&\ds_{\sigma\in S_n}\sgn(\sigma) \alpha_{1\sigma(1)}\cdots \alpha_{(k-1)\sigma(k-1)}\big(\beta_{k\sigma(k)}+\gamma_{k\sigma(k)}\big)\alpha_{(k+1)\sigma(k+1)} \cdots \alpha_{n\sigma(n)}\\
	\subseteq & \ds_{\sigma\in S_n}\sgn(\sigma) \alpha_{1\sigma(1)}\cdots \alpha_{(k-1)\sigma(k-1)}\beta_{k\sigma(k)}\alpha_{(k+1)\sigma(k+1)} \cdots \alpha_{n\sigma(n)}\\
	+&\ds_{\sigma\in S_n}\sgn(\sigma) \alpha_{1\sigma(1)}\cdots \alpha_{(k-1)\sigma(k-1)}\gamma_{k\sigma(k)}\alpha_{(k+1)\sigma(k+1)} \cdots \alpha_{n\sigma(n)}\\
	=&\det(\B)+\det(\C).
	\end{alignat*}
	
	We now assume that  $\max\limits_{\substack{1\leq i,j\leq n\\i\not=k}} R(\alpha_{ij})\leq \max\Big\{ \min\limits_{1\leq j\leq n} R(\beta_{kj}), \min\limits_{1\leq j\leq n} R(\gamma_{kj})\Big\}$. For each $\sigma\in S_n,$ let $\lambda_{\sigma}=\sgn(\sigma)\alpha_{1\sigma(1)}\cdots \alpha_{(k-1)\sigma(k-1)}\alpha_{(k+1)\sigma(k+1)} \cdots \alpha_{n\sigma(n)}$. 
	By Lemma \ref{relativeprecision} one has  $$R(\lambda_{\sigma})=		\max\limits_{\substack{1\leq i\leq n\\i\not=k}} R(\alpha_{i\sigma(i)}).$$
	From \eqref{dieukiendinhthuc} one derives that $R(\lambda_{\sigma})\leq \max\limits_{\substack{1\leq i,j\leq n\\i\not=k}} R(\alpha_{ij})\leq \max\{R(\beta_{kj}), R(\gamma_{kj})\}$ for all $1\leq j\leq n.$ By Part \eqref{tcopti3} of Theorem \ref{tcopti} we have 
	\begin{equation}\label{equality} \lambda_{\sigma}(\beta_{kj}+\gamma_{kj})=\lambda_\sigma\beta_{kj}+\lambda_\sigma \gamma_{kj},
	\end{equation}  for all $1\leq j\leq n.$
	
	If $\beta_{kj}$ and $\gamma_{kj}$ are not nearly opposite, we also have \eqref{equality}.
	This means that for all $\sigma\in S_n,$
	\begin{alignat*}{2}& \sgn(\sigma)\alpha_{1\sigma(1)}\cdots \alpha_{(k-1)\sigma(k-1)}\big(\beta_{k\sigma(k)}+\gamma_{k\sigma(k)}\big)\alpha_{(k+1)\sigma(k+1)} \cdots \alpha_{n\sigma(n)}\\
	=& \sgn(\sigma) \alpha_{1\sigma(1)}\cdots \alpha_{(k-1)\sigma(k-1)}\beta_{k\sigma(k)}\alpha_{(k+1)\sigma(k+1)} \cdots \alpha_{n\sigma(n)}\\
	+&\sgn(\sigma) \alpha_{1\sigma(1)}\cdots \alpha_{(k-1)\sigma(k-1)}\gamma_{k\sigma(k)}\alpha_{(k+1)\sigma(k+1)} \cdots \alpha_{n\sigma(n)}.
	\end{alignat*} 
	
	As a result, 
	\begin{alignat*}{2} \det&(\A)=\ds_{\sigma\in S_n}\sgn(\sigma) \alpha_{1\sigma(1)}\cdots \alpha_{n\sigma(n)}\\
	=&\ds_{\sigma\in S_n}\sgn(\sigma) \alpha_{1\sigma(1)}\cdots \alpha_{(k-1)\sigma(k-1)}\big(\beta_{k\sigma(k)}+\gamma_{k\sigma(k)}\big)\alpha_{(k+1)\sigma(k+1)} \cdots \alpha_{n\sigma(n)}\\
	= & \ds_{\sigma\in S_n}\sgn(\sigma) \alpha_{1\sigma(1)}\cdots \alpha_{(k-1)\sigma(k-1)}\beta_{k\sigma(k)}\alpha_{(k+1)\sigma(k+1)} \cdots \alpha_{n\sigma(n)}\\
	+&\ds_{\sigma\in S_n}\sgn(\sigma) \alpha_{1\sigma(1)}\cdots \alpha_{(k-1)\sigma(k-1)}\gamma_{k\sigma(k)}\alpha_{(k+1)\sigma(k+1)} \cdots \alpha_{n\sigma(n)}\\
	=&\det(\B)+\det(\C).
	\end{alignat*}
\end{proof}
Because of subdistributivity, if  $\alpha$ is an external number, one always has 	 $ \alpha\det(\A)\subseteq\det(\beta)$, where $\B=(\beta_{ij})_{n\times n}$ with $$\beta_{ij}=\begin{cases} \alpha_{ij} &\mbox{ if } i\not=k\\
\alpha\alpha_{ij}&\mbox{ if } i=k\end{cases}$$ for all $j\in \{1,\dots, n\}$. Note that the equality may not occur. For example, let $\alpha=\oslash$ and $\A=\begin{pmatrix}
1&1
\\
1&1
\end{pmatrix}$ and let $ B $ be obtained by multiplying the first row of $\A$ by $\alpha$, i.e. $\B=\begin{pmatrix}
\oslash& \oslash \\
1&1
\end{pmatrix}$. Then $\alpha\det(\A)=0\subset \det(\B)=\oslash$. However, if the relative uncertainty of $\alpha$ is less than or equal to the relative uncertainty of all entries in $\A$, equality is obtained. 

\begin{proposition}\label{scalar multiplication of determinant} 
	Let $\alpha$ be an external number and $\A=(\alpha_{ij})_{n\times n}\in \M_{n}(\E)$. Assume that $R(\alpha)\leq \min\limits_{\substack{1\leq i\leq n\\1\leq j\leq n}} R(\alpha_{ij})$. Let $k\in \{1,\dots, n\}$ and $\B=(\beta_{ij})_{n\times n}$ with $$\beta_{ij}=\begin{cases} \alpha_{ij} &\mbox{ if } i\not=k\\
	\alpha\alpha_{ij}&\mbox{ if } i=k\end{cases}$$ for all $j\in \{1,\dots, n\}$. Then $\det(\B)=\alpha\det(\A)$.
\end{proposition}
\begin{proof}
	One has 
	\begin{alignat*}{2}\det(\B)=&\ds_{\sigma\in S_n}\sgn(\sigma) \beta_{1\sigma(1)}\cdots \beta_{n\sigma(n)}\\
	=&\ds_{\sigma\in S_n}\sgn(\sigma) \alpha_{1\sigma(1)}\cdots\alpha_{(i-1)\sigma(i-1)}\alpha\alpha_{i\sigma(i)}\alpha_{(i+1)\sigma(i+1)}\cdots \alpha_{n\sigma(n)}
	\end{alignat*}
	Put $\lambda_\sigma=\alpha_{1\sigma(1)}\cdots\alpha_{(i-1)\sigma(i-1)}\alpha_{i\sigma(i)}\alpha_{(i+1)\sigma(i+1)}\cdots \alpha_{n\sigma(n)}.$\\
	Then $R(\lambda_\sigma)=\max\limits_{1\leq i\leq n} R(\alpha_{i\sigma(i)})$ by Lemma \ref{relativeprecision}. By the assumption,
	$$R(\alpha)\leq \min\limits_{1\leq i, j\leq n}R(\alpha_{ij})\leq \max\limits_{1\leq i\leq n} R(\alpha_{i\sigma(i)})=R(\lambda_\sigma)$$ for all $\sigma\in S_n.$
	By Proposition \ref{sumexternalnumber} one has
	$$\det(\B)=\alpha\ds_{\sigma\in S_n}\sgn(\sigma) \alpha_{1\sigma(1)}\cdots \alpha_{n\sigma(n)}=\alpha\det(\A).$$
\end{proof}

To study the effect of adding a multiple of one line to another, we need an adaptation of the notion of reduced matrix. We will give estimations for the determinant of a reduced matrix and its neutrix, as well as it minors. 

\begin{definition}\rm A matrix $\A=(\alpha_{ij})_{m\times n}\in \M_{m,n}(\E)$, with $\left|\overline{\alpha}\right|=1+A$ and $A\subseteq \oslash$, is called  a  {\em reduced matrix}. \index{matrix!reduced}  
\end{definition}

Let $A\in \M_{m,n}(\E)$. We denote by $M_{i_1\dots i_k,j_1\dots j_k}$ the $k\times k$ minor of $\A$ by containing only the rows $\{i_1\dots i_k\}$ and columns $\{j_1\dots j_k\}$ from $\A$. We may denote this minor also by $ \Delta_{{i_1\dots i_k,j_1\dots j_k}} $.

Reduced matrices have in each column (row) a minor of  $(n-1)^{th}$-order at least of the same order of magnitude as the determinant. 
\begin{proposition}[\cite{Jus2}]\label{lei2} Let $n\in \N$ be standard and $\A=(\alpha_{ij})_{n\times n}\in \M_{n}(\E)$ be a  reduced square matrix of order $n$. Suppose that $\Delta=\det \A$ is zeroless. Then for each $j\in \{1, \dots, n\},$ there exists  $i\in \{1, \dots,n\}$ such that $$|\Delta_{i,j}|>\oslash\Delta.$$
\end{proposition}
\begin{proof} For simplicity we prove only the case  $j=1$, the other cases are proved analogously. 
	By Proposition \ref{lei} one has $$\alpha_{11}\Delta_{1,1}-\alpha_{21}\Delta_{2,1}+\cdots+\alpha_{n1}(-1)^{n+1}\Delta_{n,1}\subseteq \Delta.$$	
	Suppose that $\Delta_{i,1}\subseteq \oslash\Delta $ for all $i=1,\dots, n.$ Because the matrix is reduced,  it holds that $|\alpha_{ij}|\leq 1+\oslash $ for all $  1\leq i, j\leq n$. So $\alpha_{i1}\Delta_{i,1}\subseteq (1+\oslash)\oslash\Delta=\oslash\Delta$ for all $i=1,\dots,n.$ Consequently, $$\alpha_{11}\Delta_{1,1}-\alpha_{21}\Delta_{2,1}+\cdots+\alpha_{n1}(-1)^{n+1}\Delta_{n,1}\subseteq \oslash\Delta.$$ So $\alpha_{11}\Delta_{1,1}-\alpha_{21}\Delta_{2,1}+\cdots+\alpha_{n1}(-1)^{n+1}\Delta_{n,1} \subseteq \Delta\cap \oslash\Delta$, a contradiction to Proposition \ref{propnam}.\ref{giao}, for $\Delta$ is zeroless.
\end{proof}

The results below give an upper bound for the minors and the corresponding neutrix parts of a reduced matrix.

\begin{proposition}\label{cdt} Let $n\in \N$ be standard and  $\A=(\alpha_{ij})_{n\times n}\in \M_{n}(\E)$ be a reduced matrix.
	Let $k\in\{1,\dots, n\}$   and   $1\leq i_1<\dots<i_k\leq n,\  1\leq j_1<\dots<j_k\leq n$. Then 
	$$\Delta_{i_1\dots i_k,j_1\dots j_k}\subset \pounds.$$
\end{proposition}
\begin{proof}
	Let $I=\{i_1,\dots, i_k\}, J=\{j_1,\dots, j_k\}$. Let $S_{k}$ be the set of all bijections $\sigma{:} \ I \rightarrow J.$ Because $\A$ is a reduced matrix, it follows that $|\alpha_{ij}|\leq 1+\oslash$ for all $1\leq i, j\leq n$. So
	\begin{alignat*}{2}|\Delta_{i_1 \dots i_k, j_1\dots j_k}| =& \left|\sum_{\sigma\in S_{k}} \sgn(\sigma) \alpha_{i_1\sigma(i_1)}\dots \alpha_{i_{k}\sigma(i_{k})}\right|\\
	\leq &\sum_{\sigma\in S_{k}} \left|\alpha_{i_1\sigma(i_1)}\right|\dots \left|\alpha_{i_{k}\sigma(i_{k})}\right|\leq \sum_{\sigma\in S_{k}} (1+\oslash)^{k}\\
	=& k!(1+\oslash).
	\end{alignat*} 
	Because $n\in \N$ is standard and $k\leq n$, it follows that $k!\leq  \pounds.$ Consequently, $k!(1+\oslash)\leq  \pounds.$ Hence $\Delta_{i_1 \dots i_k, j_1\dots j_k} \subset \pounds.$
\end{proof}

\begin{proposition}\label{danhgianeutrix} Let $n\in \N$ be standard and $\A=(\alpha_{ij})_{n\times n}\in \M_{n}(\E)$ be a reduced matrix. Let $\Delta=\det\A,$  $k\in\{1,\dots, n\}$   and   $1\leq i_1<\dots<i_k\leq n,\  1\leq j_1<\dots<j_k\leq n$. Then for all $1\leq k\leq n,$ one has
	$$N\left(\Delta_{i_1\dots i_k,j_1\dots j_k} \right)\subseteq \oA.$$ 
	In particular $N(\Delta)\subseteq \oA.$
\end{proposition}
\begin{proof} Let $I=\{i_1,\dots, i_k\}, J=\{j_1,\dots, j_k\}$. Let $S_{k}$ be the set of all bijections $\sigma : \ I \rightarrow J.$ Because $\A$ is a reduced matrix, it follows that $|\alpha_{ij}|\leq 1+\oA$ for all $1\leq i, j\leq n$, while $ \oA\subseteq \oslash $. So
	\begin{alignat*}{2}N\left( \Delta_{i_1 \dots i_k, j_1\dots j_k} \right)=& N\left(\sum_{\sigma\in S_{k}} \sgn(\sigma) \alpha_{i_1\sigma(i_1)}\dots \alpha_{i_{k}\sigma(i_{k})}\right)\\
	= &\sum_{\sigma\in S_k} N\left(\alpha_{i_1\sigma(i_1)}\cdots \alpha_{i_{k}\sigma(i_{k})}\right)\subseteq  \sum_{\sigma\in S_k} N\left((1+\oA)^k\right)\\
	=& \sum_{\sigma\in S_k} \oA=k!\oA=\oA.
	\end{alignat*} 	
	When $k=n$ we obtain that $N(\Delta)\subseteq \oA.$
\end{proof} 
Adding a scalar multiple of one row of a matrix of real numbers to another row does not change the value of the determinant. This does no longer hold for a matrix with external numbers, for we may blow up neutrices. For example, let $\A=\begin{pmatrix}
1&1\\
\oslash &1
\end{pmatrix}$ and $\omega$ be an unlimited number. Let $\B$ be the matrix which is obtained from the matrix $\A$  by adding a  multiple $\omega$ of the second row to the first one. Then  $\B=\begin{pmatrix}
1+\omega\oslash &1+\omega\\
\oslash &1
\end{pmatrix}$. We see that $\det(\A)=1+\oslash$ while $\det(\B)=\omega\oslash$, so a zeroless determinant is even transformed into a big neutrix.

We present a general property on how determinants behave under the addition of multiples of lines, which implies a condition of invariance. 
\begin{proposition}\label{addrow}
	Let $\A=(\alpha_{ij})_{n\times n}\in \M_{n}(\E)$ and $p, k\in \{1, \dots, n\}$. Let $\A'=(\alpha'_{ij})_{n\times n}\in \M_{n}(\E)$ where for all $j\in \{1,\dots, n\}$ we define  $$\alpha'_{ij}=\begin{cases}
	\alpha_{ij} & \mbox{ if } i\not=k\\
	\alpha_{kj} +\lambda \alpha_{pj} &\mbox{ if } i=k
	\end{cases}.$$ i.e., we add a multiple $ \lambda\in \E $ of the $p^{th}$ row to the $k^{th}$ row. Assume that $R(\lambda)\leq \min\limits_{\substack{1\leq i\leq n\\1\leq j\leq n}} R(\alpha_{ij})$ and   $|\overline{\alpha}|=\max\limits_{\substack{1\leq i\leq n\\1\leq j\leq n}}|\alpha_{ij}|$ is zeroless. Then $$\det(\A')\subseteq \det(\A)+\lambda \overline{\alpha}^{n-1}\overline{A}.$$
 As a result, if $\lambda \overline{\alpha}^{n-1}\overline{A}\subseteq N(\det(\A))$ then $\det(\A')=\det(\A)$.
\end{proposition}

\begin{proof}
	Let $ \A'' $ be obtained from $ \A $ by copying line $ p $ to line $ k $; then  $\A'' $ takes the form
	$$\A''=   \begin{pmatrix}
	\alpha_{11}& \alpha_{12} & \cdots &\alpha_{1n}\\
	\vdots &\vdots& \ddots & \vdots\\
	\alpha_{p1} &\alpha_{p2} &\cdots & \alpha_{pn}\\
	\vdots &\vdots& \ddots & \vdots\\
	\alpha_{p1} &\alpha_{p2} &\cdots &\alpha_{pn}\\
	\vdots &\vdots& \ddots & \vdots\\
	\alpha_{n1} &\alpha_{n2} &\cdots& \alpha_{nn} 
	\end{pmatrix}.$$ 
	Because $ |\overline{\alpha}| $ is zeroless, we may choose a representative $ a\in\alpha $ such that $ |\overline{\alpha_{ij}}/a|\leq 1+\oslash$ for all $ 1\leq i,j\leq n $. Let $ \mathcal{R} $ be obtained from $ \A'' $ by dividing every entry by $ a $, then $ \mathcal{R} $ is a reduced matrix. 
	By Proposition \ref{plus of determinant} and \ref{scalar multiplication of determinant} we have 
	\begin{equation}\label{add determinant 1} 
	\det(\A')\subseteq\det(\A)+\lambda\det(\A'').	
	\end{equation}
	Now $ \det(\A'') $ is a neutrix since $ \A'' $ has two identical rows. Also, by Proposition \ref{scalar multiplication of determinant} and \ref{danhgianeutrix}
	\begin{equation} \label{add determinant 2} \det(\A'')=\overline{a}^n \det(\mathcal{R})  
	\subseteq  \overline{a}^n \overline{A}/\overline{a}=\overline{a}^{n-1}\overline{A}=\overline{\alpha}^{n-1}\overline{A}.
	\end{equation}
The last equality holds because $\overline{\alpha}$ is zeroless.  
	From \eqref{add determinant 1} and \eqref{add determinant 2} we have $\det(\A')\subseteq \det(\A)+\lambda \overline{\alpha}^{n-1}\overline{A}$, hence $ \det(\A')= \det(\A) $ if $\lambda \overline{\alpha}^{n-1}\overline{A}\subseteq N(\det(\A))$.
\end{proof}
Observe that the first condition of Proposition~\ref{addrow} is automatically satisfied if $ \lambda\in \R $.

Classically we use Gauss-Jordan elimination to transform a determinant into a determinant of a triangular matrix, and then the determinant is the product of the elements on the diagonal. In the context of external numbers the usual techniques of Gauss-Jordan elimination generate neutrices instead of zeros, and sometimes the determinants were modified by a neutrix. 
Also to obtain the determinant of a triangular matrix it may be needed to add a neutrix to the product of the entries on the diagonal.
\begin{definition}
	Let $\A=(\alpha_{ij})_{n\times n}\in \M_n(\E)$. The matrix $\A$ is called \emph{upper triangular}  if $\alpha_{ij}$ is a neutrix for all $1\leq j<i\leq n$. The matrix $\A$ is called \emph{lower triangular} if $\alpha_{ij}$ is a neutrix for all $1\leq i<j\leq n.$ An upper triangular or lower triangular matrix is called a \emph{triangular matrix}.
\end{definition}
A simple triangular matrix such that its determinant involves a neutrix which even makes the matrix singular is given by the following example. Let $\A=\begin{pmatrix}
1&\oslash\\
\omega & 1
\end{pmatrix}$, where $\omega$ is an unlimited number. Then $1=1\cdot  1\not=\det(\A)= \omega\oslash$. Next proposition gives an upper bound for such neutrices. 
\begin{proposition}
	Let $\A=(\alpha_{ij})_{n\times n}$ be a triangular matrix. Assume that $\overline{\alpha}$ is zeroless.  If $\A$ is reduced,  
	\begin{equation}\label{det triangreduced} \det(\A)\subseteq \alpha_{11}\alpha_{22}\cdots\alpha_{nn}+\overline{A}.
	\end{equation}
	In general 
	\begin{equation*}%\label{det triangular} 
	\det(\A)\subseteq \alpha_{11}\alpha_{22}\cdots\alpha_{nn}+\overline{\alpha}^{n-1}\overline{A}.
	\end{equation*} 
	As a result, if $\overline{\alpha}^{n-1}\overline{A}\subseteq N(\alpha_{11}\alpha_{22}\cdots\alpha_{nn})$, then $\det(\A)=\alpha_{11}\alpha_{22}\cdots\alpha_{nn}. $ 
\end{proposition}
\begin{proof} Without loss of generality, we assume that $\A$ is a upper triangular matrix. We have
	\begin{alignat}{2} \label{determinant of triangular matrix} \det(\A)=&\sum\limits_{\sigma\in S_n}\sgn(\sigma) \alpha_{1\sigma(1)}\cdots\alpha_{n\sigma(n)} \notag\\
	=&\alpha_{11}\cdots\alpha_{nn} +\sum\limits_{\substack{\sigma\in S_n\\\exists i\in \{1,\dots n\},\sigma(i)\not =i}} \sgn(\sigma)\alpha_{1\sigma(1)}\cdots \alpha_{n\sigma(n)}.
	\end{alignat} 
	We consider two cases. First we assume that $\A$ is a reduced matrix. For $\sigma\in S_n$ such that there exists $i\in \{1,\dots,n\}$, $i\not=\sigma(i)$, it follows that there exists $k\in \{1,\dots, n\}$ such that $k>\sigma(k)$. Then $\alpha_{k\sigma(k)}\equiv A_{k\sigma(k)}$ is a neutrix. As a consequence, $\alpha_{1\sigma(1)}\cdots\alpha_{n\sigma(n)}$ is a neutrix. Also $|\alpha_{ij}|\leq 1+\oslash$ for all $1\leq i, j\leq n$, hence $\alpha_{1\sigma(1)}\cdots\alpha_{n\sigma(n)}\subseteq A_{k\sigma(k)}\subseteq \overline{A}.$
	
Because $|\overline{\alpha}|=1+\oslash$, we derive from \eqref{determinant of triangular matrix} that $\det(\A)\subseteq \alpha_{11}\cdots\alpha_{nn}+\overline{A}$. Using \eqref{det triangreduced}, we find that $\det(\A)\subseteq \alpha_{11}\cdots\alpha_{nn}+\overline{\alpha}^{n-1}\overline{A}$.

Second, assume that $\A$ is an arbitrary matrix such that $\overline{\alpha}$ is zeroless. Let $\overline{a}\in \overline{\alpha}$ and  $\A'=(\alpha'_{ij})$ with $\alpha'_{ij}=\alpha_{ij}/\overline{a}$ for all $ 1\leq i,j,\leq n $. Then $\A'$ is a reduced upper triangular matrix. Also  $\det(\A)=\overline{a}^n \det(\A')\subseteq \overline{a}^{n}( \alpha'_{11}\cdots\alpha'_{nn}+\overline{A}/\overline{a})=\alpha_{11}\cdots\alpha_{nn}+\overline{a}^{n-1}\overline{A}=\alpha_{11}\cdots\alpha_{nn}+\overline{\alpha}^{n-1}\overline{A}.$
\end{proof}

\section{Inverse matrices% with external numbers
}\label{inverse matrix}
The additive inverse of an external number $ \alpha $ is defined up to a neutrix, for $ \alpha-\alpha=N(\alpha) $. Proposition~\ref{algmat} shows that the additive inverse of a matrix of external numbers exists up to a neutricial matrix. When $ \alpha $ is zeroless, the  multiplicative inverse satisfies $ \alpha/\alpha=1+R(\alpha) $ with $ R(\alpha)\subseteq \oslash $. We define the multiplicative inverse of a matrix of external numbers also with respect to a neutrix contained in $ \oslash $. This neutrix is an upper bound for the precision that can be obtained and the (not unique) inverse is defined in terms of inclusion. We recall that \emph{flexible systems}, i.e. systems of linear equations with coefficients and constant term given by external numbers, are also defined for inclusions \cite{Jus}. 

The relationship between invertible matrices and non-singular matrices (matrices with zeroless determinant) is investigated, as well as the possibility to determine inverses with the help of cofactors. This happens to be possible under a quite general condition, already present in \cite{Jus} when solving non-singular systems; in particular the determinant should not be too small.
\begin{definition}
Let $\A=(\alpha_{ij})_{n\times n}\in \M_n(\E)$. 	The matrix $\A$ is called \emph{non-singular} if $\det(\A)$ is zeroless. Otherwise we call it \emph{singular}. 
\end{definition}  
\begin{definition}
	Let $\A\in \M_n(\E)$ be a square matrix, $N\subseteq \oslash$ be a neutrix and $\I_n(N)=(\delta_{ij})\in \M_n(E)$ with $\delta_{ij}=\begin{cases} 1+N &\mbox{ if }i=j\\
	N &\mbox{ if } i\not=j\end{cases}$ for all $1\leq i,j\leq n$. The matrix $\A$ is said to be \emph{ invertible with respect to $N$} if there exists a square matrix $\B=(\beta_{ij})_{n\times n}$   such that  $$\begin{cases} \A\B\subseteq \mathcal{I}_n(N),\\
	\B\A\subseteq \mathcal{I}_n(N).
	\end{cases} $$ 
	 Then $\B$ is called an \emph{inverse matrix} of $\A$ with respect to $N$ and denoted by $\A^{-1}_N$.
\end{definition}
It is clear that if $\A$ is invertible with respect to $N\subseteq \oslash$, it is invertible with respect to every neutrix $M$ with $N\subseteq M\subseteq \oslash$. In case $\A$ is a real square matrix, the inverse matrix of $\A$ with respect to $0$ becomes the classical one and we simply write $\A^{-1}$. 

The matrix $\dfrac{1}{\det(\A)}\C^T$, where $\C$ is the cofactor matrix of $\A$, is not always an inverse matrix of $\A$ with respect to a neutrix, even if $\A$ is a non-singular matrix. Indeed, let $\eps>0$ be infinitesimal and  $\A=\begin{pmatrix}
\eps & \oslash\\
0& 1
\end{pmatrix}$. Then $\det(\A)=\eps$ is zeroless, so $\A$ is non-singular. We have $\B=\dfrac{1}{\det(\A)}\C^T=\begin{pmatrix}
\frac{1}{\eps}& 0\\
\frac{\oslash}{\eps} &1
\end{pmatrix}$. This implies that $\A.\B=\begin{pmatrix}
\eps & \oslash\\
0& 1
\end{pmatrix}.\begin{pmatrix}
\dfrac{1}{\eps}& 0\\
\frac{\oslash}{\eps} &1
\end{pmatrix}=\begin{pmatrix}
1+\frac{\oslash}{\eps} & \oslash\\
\frac{\oslash}{\eps} & 1
\end{pmatrix}=\begin{pmatrix}
\frac{\oslash}{\eps} & \oslash\\
\frac{\oslash}{\eps} & 1
\end{pmatrix}\not\subseteq\mathcal{I}_2(N)$ for all $N\subseteq \oslash$.  Hence $\B$ is not an inverse matrix of $\A$.

%Note that in case $\A\in \M_n(\R)$ then the definition above coincides with the one in the classical algebra. 
% However, the definition of inverse matrix satisfies that the neutrix part of all entries in the inverse matrix is not larger than the maximal neutrix in given matrix. 

%\begin{definition} \rm 
%	Let  $\A=(\alpha_{ij})\in \M_{m,n}(\E)$. A matrix $P=(a_{ij})\in \M_{m,n}(\R)$, with $a_{ij}\in \alpha_{ij}$ for all $1\leq i\leq m, 1\leq j\leq n$, is called  a \emph{representative matrix} of  matrix $\A$. 
%\end{definition} 

\begin{theorem}
	Let $\A=(\alpha_{ij})_{2\times 2}\in \M_{2}(\E)$ be an invertible matrix with respect to a neutrix $N\subseteq \oslash$. Then $\A$ is non-singular. 
\end{theorem}
\begin{proof}
	Suppose that $\A$ is singular. Then $0\in \det(\A)$. By \eqref{detrep} there exists a representative matrix $P$ of $\A$ such that $\det(P)=0$. Let $Q$ be a representative matrix of $\B$. Then 
	\begin{equation}\label{determinant0}
	\det(PQ)=\det(P)\det(Q)=0.
	\end{equation}  
	On the other hand, one has $\A\B\subseteq \mathcal{I}_2(N)$. Now $PQ$ is a representative matrix of $\mathcal{I}_2(N)$, so $\det(PQ)\not=0$, contradicting \eqref{determinant0}. Hence $\A$ is non-singular. 
\end{proof}
The result above does not hold any more for $n>2$. For example, the matrix $\A=\begin{pmatrix}
1+\oslash & 0&0\\
0 &1&1+\varepsilon\\
0& 1& 1
\end{pmatrix}$ with $\varepsilon \simeq 0,\eps\not=0$ of Example~\ref{Sarrus} is invertible with respect to $\oslash$, but it is singular with $ \det(\A)=\oslash $. But being a $ 3\times 3 -$matrix, it does not need to have a singular matrix of representatives, which happens indeed.

If the matrix $ \A $ is reduced, a converse for holds if  $ \det(\A) $ is not so small to be an absorber of $\overline{A}$. A general condition will be given in terms of the relative uncertainty of \cite{Jus}.
\begin{definition}
Let $\A=(\alpha_{ij})_{n\times n}\in \M_{n}(\E)$ be such that $\overline{\alpha}$ is zeroless. Then $ R(\A)\equiv \det(\A)/\overline{\alpha}^n$ is called the \emph{relative uncertainty} of $ \A $.
\end{definition}
\begin{theorem} Let $\A=(\alpha_{ij})_{n\times n}\in \M_{n}(\E)$ be a non-singular matrix. Assume that
\begin{enumerate}
	\item $\overline{\alpha}$ is zeroless.
	\item $R(\A)$ is not an absorber of $\overline{A}.$
\end{enumerate}  Then $\A$ is invertible with respect to $\dfrac{\overline{A}}{\overline{\alpha}}$ and $\dfrac{1}{\det(\A)} \C^T$ is an inverse matrix with respect to $\dfrac{\overline{A}}{\overline{\alpha}}$ of $\A$, where $\C$ is the cofactor matrix of $\A$. 
\end{theorem}
\begin{proof}  Note that $\overline{A}/\overline{\alpha}\subseteq \oslash$, because $\overline{\alpha}$ is zeroless.
	
 We first assume that $\A$ is a reduced, non-singular matrix. Let $\A=(\alpha_{ij})_{n\times n}$ with $\alpha_{ij}=a_{ij}+A_{ij}$. Let $P=(a_{ij})_{n\times n}$, $ K=(A_{ij})_{n\times n}$ and $\Delta=\det(\A)=d+D$ with $d=\det(P)\neq 0$. Let $ Q=(b_{ij})_{n\times n} $ be the inverse matrix of $ P $, with $ R=(c_{ij})_{n\times n} $ the matrix of cofactors, meaning that always $ b_{ij}=\frac{c^T_{ij}}{d} $. Then the cofactor matrix is of the form $\C =(c_{ij}+C_{ij})_{n\times n}\equiv (\gamma_{ij})_{n\times n}$, and we define $ M=(C_{ij})_{n\times n}$ and $\B=\dfrac{1}{\det(\A)}\C^T=(b_{ij}+B_{ij})_{n\times n}$, where 
 $ B_{ij}=\frac{1}{d}\left(C^T_{ij}+\frac{\gamma_{ij}^TD}{d}\right)$ for all $1\leq i, j\leq n$. Let $L=(B_{ij})_{n\times n}$ and let  $ I_n $ be the identity matrix of order $n$. 
 
 We show that $B_{ij}\subseteq \overline{A}\subseteq \oslash$ for all $1\leq i, j\leq n$. Observe that $ D\subseteq \overline{A} $ and $C_{ij}\subseteq \overline{A} $ for all $1\leq i, j\leq n$ by Lemma \ref{danhgianeutrix}, and $\gamma_{ij}\subseteq \pounds$ for all $1\leq i, j\leq n$ by Proposition \ref{cdt}.  So $B_{ij}\subseteq \frac{1}{d}\left(\overline{A}+\frac{\overline{A}}{d}\right)=\frac{\overline{A}}{d}+\frac{\overline{A}}{d^{2}}$ for all $1\leq i, j\leq n.$  Also  $\det(\A)/\overline{\alpha}^n =\det(\A)$ is not an absorber of $\overline{A}$, so neither is $ d $, and therefore $d\overline{A}=\overline{A}=\frac{\overline{A}}{d}$. Consequently $B_{ij}\subseteq \overline{A}\subseteq \oslash$ for all $1\leq i, j\leq n.$

Next, we prove that \begin{equation}\label{inver} 
N(\A\B)=PL+KL+QK\subseteq (\overline{A})_{n\times n}\subseteq (\oslash)_{n\times n}. 
\end{equation}
Indeed, since $P\subseteq (\pounds)_{n\times n}$ and $L\subseteq (\overline{A})_{n\times n}$, we derive that 
\begin{equation}\label{inver1} PL\subseteq (\pounds)_{n\times n}(\overline{A})_{n\times n}=(\overline{A})_{n\times n}. 
\end{equation} 
Also $K\subseteq (\overline{A})_{n\times n}$, which implies that 
\begin{equation*}%\label{inver2}
 KL\subseteq (\overline{A})_{n\times n}(\overline{A})_{n\times n}\subseteq (\overline{A})_{n\times n}.
\end{equation*} 
In addition, \begin{equation} \label{inver3} KQ=K\dfrac{1}{d} (c^T_{ij})_{n\times n}\subseteq \dfrac{1}{d}(\overline{A})_{n\times n}(\pounds)_{n\times n}= \dfrac{1}{d}(\overline{A})_{n\times n}=(\overline{A})_{n\times n}.
\end{equation}  
Then \eqref{inver} follows from \eqref{inver1}-\eqref{inver3}. 

As a consequence, we have   $\A\B=PQ+PL+KQ+KL\subseteq I_n+(\overline{A})_{n\times n}=\I_n(N)$. Similarly, we have $\B\A\subseteq \mathcal{I}_n(N)$. Hence $\B=\dfrac{1}{\det(\A)}\C^T$ is an inverse matrix of $\A$ with respect to $A$.

We now assume that  $\A=(\alpha_{ij})_{n\times n}\in \M_n(\E)$ is an arbitrary non-singular matrix such that $\overline{\alpha}$ is zeroless. Then $\A=\overline{a}\G$ where $\G=(\alpha_{ij}/\overline{a})_{n\times n}\equiv (\eta_{ij})$ is the reduced matrix and $\overline{a}\in \overline{\alpha}$. Because $\A$ is non-singular, the matrix $\G$ is non-singular. Also $\overline{\eta}=\overline{\alpha}/\overline{a}$ is zeroless. Let $\eta_{ij}=g_{ij}+G_{ij}$ for all $1\leq i, j\leq n$ and $\overline{G}=\max\limits_{1\leq i,j\leq n} G_{ij}=\dfrac{\overline{A}}{\overline{a}}$. Also $R(\A) $ is not an absorber of $\overline{A}$, hence
 \begin{equation}\label{maxneu} \overline{G}=\frac{\overline{A}}{\overline{a}}\subseteq \frac{1}{\overline{a}}\left(\frac{\det(\A)}{\overline{\alpha}^n}\overline{A}\right)= \frac{1}{\overline{a}}\left(\frac{\det(\A)}{\overline{a}^n}\overline{A}\right)=\det(\G)\frac{\overline{A}}{\overline{a}}=\det(\G)\overline{G},.
\end{equation}
implying that $\det(\G)$ is not an absorber of $\overline{G}.$ Since $ \G $ is reduced, by the above argument  $\G^{-1}=\dfrac{1}{\det(\G)} \mathcal{D}^T$ is an inverse matrix of $\G$ with respect to $\overline{G}/\overline{\eta}=\overline{G}$, where $\mathcal{D}$ is the cofactor matrix of $\G$. Let $\mathcal{H}=\dfrac{1}{\overline{a}}\G^{-1}=(h_{ij}+H_{ij})$. Then $\mathcal{H}$ is an inverse matrix of $\A$ with respect to $\overline{G}$. Indeed, $\A\dfrac{1}{\overline{a}}\G^{-1}=\overline{a} \G \dfrac{1}{\overline{a}}\G^{-1}=\G\G^{-1}\subseteq \mathcal{I}_n(\overline{G}).$ Similarly, we have $\dfrac{1}{\overline{a}}\G^{-1}\A\subseteq \mathcal{I}_n(\overline{G}).$ This means that $\dfrac{1}{\overline{a}}\G^{-1}$ is an inverse matrix of $\A$ with respect to $\overline{G}=\dfrac{\overline{A}}{\overline{a}}$. Note that $\dfrac{1}{\overline{a}}\G^{-1}=\dfrac{1}{\det{\A}}\C^T$ where $\C$ is the cofactor matrix of $\A$.

Combining, we conclude that $\A^{-1}_A=\dfrac{1}{\det{\A}}C^T$. 
\end{proof} 
In case all conditions in Theorem above hold, the choice of the representative matrix $P$ of $\A$ is arbitrary and $P^{-1}$ and is always a representative of $\A^{-1}$. The final proposition of this section is an obvious consequence of the fact that $(P^{-1})^{-1}=P$.

\begin{proposition}
	Let $\A=(\alpha_{ij})_{n\times n}\in \M_n(\E)$ be invertible matrix with respect to a neutrix $N$ and let $\left(\A^{-1}\right)_N$ be an inverse matrix with respect to $N$ of $\A$. Then $\left(\A^{-1}\right)_N$ is invertible with respect to $N$ and $\A$ is an inverse matrix of $ \A^{-1} $ with respect to $N$. %\label{inverse1}
%		\item \label{inverse2} $\A\B$ is invertible matrix with respect to $A+B$ and $(\A\B)^{-1}_{A+B}=\B^{-1}_B\A^{-1}_A.$
%	\end{enumerate}
\end{proposition} 

\section{Linear dependence and independence% of vectors with external numbers
}\label{section3c3} 

In this section we will study sets of vectors with external numbers. We will always suppose that the sets are finite and have a standard cardinality. A generalized notion of linear independence is given. We present some characterizations and verify that several common properties of independence continue to hold. 

We start by introducing some useful notions for external vectors. 

\begin{definition}
Let $\beta=(\beta_1,\dots, \beta_m)\in \E^n$. A vector $b=(b_1,\dots, b_n)$, where $b_i\in \beta_i$ for $1\leq i\leq n$, is said to be a \emph{representative} \index{vector!representative} of $\beta$. If $\overline{\beta} $ is a neutrix, $\beta$ is called  an {\em upper neutrix vector}.  

Let $ A_1,\dots, A_n $ be neutrices. Then $A\equiv(A_1,\dots, A_n)$ is called a {\em neutrix vector} and for each $1\leq k\leq n$, a vector of the form $$e^{(k)}_A=(A_1,\dots, A_{k-1}, 1+A_k, A_{k+1}, \dots, A_n)$$  is called a \emph{near unit vector}.
\end{definition}

For example, the vector $\beta=\left(\eps+\eps^2\oslash, \oslash, \eps+\eps^2\pounds\right)$ is an upper neutrix vector since $\overline{\beta}=\oslash$ is  a neutrix and the vector $\beta_1=\left(1+\eps^2\oslash, \oslash, 2+\eps\pounds\right)$ is not an upper neutrix vector, because $\overline{\beta}=2+\eps\pounds$ is zeroless.

Neutrix vectors can be seen as generalizations of the zero vector, and they are used in the following definition of linear dependence.

\begin{definition}\label{defind}   
	A set of vectors $V=\{\alpha_1,\dots, \alpha_m\}$ where $\alpha_i\in \E^n$ for $1\leq i\leq m$  is called {\em linearly dependent} if there exist real numbers $t_1,t_2,...,t_m\in \mathbb{R}$, at least one of them being non-zero, and a neutrix vector $A$ such that  $$t_1\alpha_1+t_2 \alpha_2+\cdots+t_m\alpha_m=A.$$	
	Otherwise, the set $V$ is called {\em linearly independent}.					
\end{definition}			
In case $\{\alpha_1,\dots, \alpha_m\}\subset \R^m$, the notions coincide with those in the conventional algebra.	

From definition~\ref{defind} we easily obtain the following characterization for linear independence. 	
\begin{proposition}\label{dauhieudoclaptuyentinh}
	A set $V=\{\alpha_1, \cdots, \alpha_m\}$ of vectors   in $\E^{n}$ is linearly independent if and only if the equality $t_1\alpha_1+t_2\alpha_2+\cdots+t_m\alpha_m=A$, where $A$ is a neutrix vector, implies $t_1=\cdots=t_m=0$ and $A$ is the null vector.
\end{proposition}						
\begin{example}\rm Let $\epsilon >0$ be infinitesimal. 
	Then the  vectors $\alpha_1=(1+\oslash,\epsilon \oslash, -2+\epsilon\pounds), \ \alpha_2=(-2+\oslash,\epsilon\pounds, 4+\epsilon\pounds)$ in $\E^3$ are linearly dependent, since $2\alpha_1+\alpha_2=(\oslash,\epsilon\pounds,\epsilon\pounds)$ is a neutrix vector.
\end{example}
\begin{example}\rm 
	The vectors $\alpha_1=(1+\oslash,\epsilon \oslash),\, \alpha_2=(\oslash,1+\epsilon\pounds)$ with $\eps>0$ in $\E^2$ are linearly independent. Indeed, let $t_1,t_2\in \R$  and  $A=(A_1, A_2)$ is a neutrix vector such that $t_1\alpha_1+t_2\alpha_2=A.$ Then there are vectors $x_1=(1+\eta,\eps\zeta)\in \alpha_1$ and $x_2=(\vartheta,1+\epsilon \lambda)\in \alpha_2$, where $\eta,\zeta,\vartheta $ are infinitesimal and $\lambda  $ is limited, such that $t_1x_1+t_2x_2=0.$ It is equivalent to the system
	$$\begin{cases} t_1 (1+\eta)+t_2\vartheta=0\\
	t_1\zeta+t_2(1+\epsilon\lambda)=0.
	\end{cases}$$
Then $ t_1=t_2=0$, because $ \det\begin{pmatrix}
1+\eta &\vartheta \\
	\zeta& 1+\epsilon\lambda
	\end{pmatrix}\neq 0  $, and $t_1\alpha_1+t_2 \alpha_2=0$.	Hence the  vectors $\alpha_1, \alpha_2$ are linearly independent.
\end{example}

The next theorem characterizes linearly independence and dependence of vectors in $\mathbb{E}^n$ via representatives.
\begin{theorem}\label{md12} Let  $$V=\{\xi_1=(\xi_{11},\dots,\xi_{1n}),\xi_2=(\xi_{21},\dots,\xi_{2n}),\dots,\xi_m=(\xi_{m1},\dots,\xi_{mn})\}\subset\mathbb{E}^n$$ be a set of vectors, with $\xi_{ij}=a_{ij}+A_{ij}$ for all $1\leq i\leq m$ and $1\leq j\leq n$. Then 
	\begin{enumerate}
		\item \label{md12i} The set $V$ of vectors in $\E^n$ is linearly dependent if and only if for all $1\leq i\leq m$, there exist representatives $x_i=(x_{i1}, \dots, x_{in})\in\R^{n}$ of $\xi_i$ such that  $x_1,\dots,x_m$ are linearly dependent.
		\item \label{md12ii} The set $V$ of vectors in $\E^n$  is linearly independent if and only if every set  $\{x_1,\dots,x_m\} $ of vectors in $\R^n$, where $x_i\in \xi_i$ for $1\leq i\leq m$,   is linearly independent.
	\end{enumerate}							
\end{theorem}							

\begin{proof}
	\ref{md12i}. Suppose that  the vectors $\xi_1,\dots,\xi_m$ are linearly dependent. By the definition, there exist real numbers $t_1,\dots,t_m$, at least one of them being non-zero,  and a neutrix vector $A=(A_1,\dots, A_n)$ such that $$t_1\xi_1+t_2 \xi_2+\cdots+t_m\xi_m=A.$$
	Consequently $(0,...,0)\in t_1\xi_1+t_2 \xi_2+\cdots+t_m\xi_m.$ Hence there exist vectors $x_i\in\xi_i , i=1,...,m$ such that $t_1x_1+t_2 x_2+\cdots+t_m x_m=0$. That is, the set  $\{x_1,...,x_m\}$ is linearly dependent.
	
	Conversely, suppose that there exists  a linearly dependent set  of vectors $V'=\{x_1,...,x_m\}\subset \mathbb{R}^n$, with $x_i\in\xi_i$  for $1\leq i\leq m$. For $1\leq i\leq m$, let $x_i=(x_{i1},...,x_{in})$  and $\xi_{ij}=x_{ij}+X_{ij}$, where $j\in\{1,...,n\}$. There exist real numbers $t_1,...,t_m$, at least one of them being non-zero, such that $t_1x_1+t_2 x_2+\cdots+t_m x_m=0$. Let $x_i=(x_{i1},...,x_{in})$ for $1\leq i\leq m$. Then  
	\begin{equation*}%\label{bdtt4.1}
	t_1x_{1j}+\cdots+t_mx_{mj}=0\mbox{ for all $j\in\{1,...,n\}$.}
	\end{equation*}
	Then
	\begin{alignat*}{2}
	t_1\xi_{1j}+\cdots+t_m\xi_{mj}= & t_1(x_{1j}+X_{1j})+\cdots+t_m(x_{mj}+X_{mj})\\
	= & t_1x_{1j}+\cdots+t_mx_{mj}+t_1X_{1j}+\cdots+t_m X_{mj}\\
	= &t_1X_{1j}+\cdots+t_m X_{mj}\equiv A_j,
	\end{alignat*}
	where $ A_j $ is a neutrix for all $j\in\{1,...,n\}.$ Hence  $\{\xi_1,...,\xi_m\}$ is linearly dependent.
	
	\ref{md12ii}. This follows directly from Part \ref{md12i}, by contraposition.

\end{proof}						
Observe that  a  set of linearly dependent vectors  may have a set of linearly independent representative vectors. 
\begin{example}\rm 
	Let $\eps>0$ be infinitesimal. Consider the set of vectors  $$\left\{ \xi_1=(\oslash, \oslash), 	\xi_2=(0, \eps)\right\}.$$ Then $\{\xi_1, \xi_2\}$ is linearly dependent, since $\xi_1+\xi_2=(\oslash, \oslash)$. Now we take $x_1=(\eps, 0)\in \xi_1$ and $x_2=\xi_2$. Then $\{\xi_1, \xi_2\}$ is linearly independent. 
\end{example} 
Below some elementary properties of linear dependence and independence are generalized to neutrix vectors. The proofs are obvious, or readily obtained by going to representative vectors.
\begin{proposition}\label{md16} Let $S=\{\xi_1,\cdots,\xi_m\}$ be a set of vectors in $\E^n$ and $k\in \N$ be standard.
	\begin{enumerate}
		\item If $ S $ contains a neutrix vector is linearly dependent.	
		\item \label{md22}If $m>n$ the set $ S $ is linearly dependent.
		\item \label{indei} If the set $S$  is linearly dependent, any set of $k$ vectors including $S$ is linearly dependent.
		\item \label{idenii} If the set $S$ is linearly independent, any set of vectors included in $S$ is linearly independent.
	\end{enumerate} 
\end{proposition}

\section{Notions of rank}\label{section4c3} 

We define the rank of a set of vectors as usual in terms of the maximal cardinality of linearly independent subsets. 

Three notions of rank of a matrix over $\E$ are given, in the form of the rank of the set of row vectors, a rank based on minors and a rank based both on the minors and the rank of a representative matrix. In general, these three notions do not match. Conditions for the equality of the ranks are presented. 
\begin{definition}\label{rank of vector} 
	Let $V=\{\xi_1,\dots, \xi_m\}$ be a set of vectors in $\E^n$. The maximal cardinality of  linearly independent subsets $ V'\subseteq V $ is called the \emph{rank} of the given set of vectors.
\end{definition}

\begin{definition}\label{hangmatran}\rm   Let $\A=(\alpha_{ij})$ be  an $m\times n$ matrix over $\E$. 
	\begin{enumerate}
		\item 	\label{row} The \emph{row-rank} of $A$ is the rank of the set of its row vectors and denoted by $r(\A)$, corresponding to the common notation rank for sets of real vectors. 
		\item 	\label{minor} The \emph{minor-rank} of  $\A$  is the largest natural number $m$ such that there exists a zeroless  minor of order $m$ of $\A$. Then we write  $\mr(\A)=m$.	
		\item \label{strict} The minor rank of $\A$ is called a \emph{strict rank}, if there exists a representative matrix $\hat{\A}$ of $\A$ such that $r(\hat{\A})=\mr(\A)$. We denote the strict rank by $\sr(\A)$.
	\end{enumerate}
\end{definition}

\begin{example} \rm Let 
	$$\A=\begin{pmatrix}
	1+\oslash & 2+\oslash & -1+\eps\pounds\\
	-2& -4+\eps & 2+\eps\oslash
	\end{pmatrix}.$$ 
	Then $ M_{12,12}=M_{12,13}=	M_{12,23}=\oslash $, while $ M_{1,1}=1+\oslash $ is zeroless. Hence $\mr(\A)=1$. It follows from the equality
	$$2(	1+\oslash, 2+\oslash, -1+\eps\pounds)+	(-2, -4+\eps,  2+\eps\oslash)=(\oslash,\oslash,\eps\pounds)
	$$ 
	that $ r(\A)=1 $. The matrix of representatives 
	$$\hat{\A}=\begin{pmatrix}
	1& 2 & -1\\
	-2& -4 & 2
	\end{pmatrix}$$
	has rank $ 1 $. Hence also $ \sr(\A)=1 $. 
\end{example}

\begin{example}\label{phanvidu} 
	The matrix $$\A=\begin{pmatrix}
	1+\oslash & 0&0\\
	0 &1&1+\varepsilon\\
	0& 1& 1
	\end{pmatrix},$$ where $\varepsilon\simeq 0,\varepsilon\neq 0$,
	of Example~\ref{Sarrus} shows that the strict rank is not always defined, and also that it is possible that the minor rank is less than the row rank. Let $\alpha_1, \alpha_2, \alpha_3$ be the row vectors of $\A$.  
	We have $\det(\A)=\oslash$, but 
	$\begin{pmatrix}		
	1&1+\varepsilon\\
	1& 1
	\end{pmatrix}$ 
	is a non-singular minor. Hence $ \mr(\A)=2 $. On the other hand, let $x_{1}=(1+\varepsilon',0,0)  $ be a representative of $ \alpha_{1} $, where $\varepsilon'\in \oslash$. Then $$\det\begin{pmatrix}
	1+\varepsilon' & 0&0\\
	0 &1&1+\varepsilon\\
	0& 1& 1
	\end{pmatrix}=-\varepsilon-\varepsilon\varepsilon'\not=0$$  It  follows that $\{x_1,\alpha_2, \alpha_3\}$  is linearly independent. Hence every matrix of representatives necessarily has rank $ 3 $. This means that the strict rank of $ \A $ is not well-defined. 
	
	Also, by Theorem \ref{md12} the set of vectors $\{\alpha_1, \alpha_{2}, \alpha_{3}\}$ is linearly independent. As a consequence $ r(\A)=3 $ and $ r(\A)>\mr(\A) $.
\end{example}

\begin{definition} \rm Let $\xi_i=(\alpha_{i1},\dots, \alpha_{in})\in \E^n, 1\leq i \leq m.$ The  matrix  $$\A=\begin{pmatrix}
	\alpha_{11} & \alpha_{12}& \cdots & \alpha_{1n}\\
	\vdots & \vdots& \ddots & \vdots \\
	\alpha_{m1} & \alpha_{m2} & \cdots & \alpha_{mn}
	\end{pmatrix}$$ is called the {\em coordinate matrix} of the given vectors and is denoted  by $[\xi_1, \dots, \xi_m]^T$. 
\end{definition}

We show that if the coordinate matrix of a set $ V $ of $n$ vectors in $ \mathbb{E}^n$ is non-singular, the set $ V $ is linearly independent, but we already saw that the converse is not true. The converse holds however for $ n\leq 2 $, as a consequence of formula~\eqref{detrep}.

\begin{theorem} \label{md14} Let  $V=\{\alpha_1,\cdots,\alpha_n\}$ be a set of $n$ vectors in $ \mathbb{E}^n$, where $\alpha_i=(\alpha_{i1},\dots, \alpha_{in})$ for $1\leq i\leq n$. Assume that the coordinate matrix is non-singular. Then $V  $ is linearly independent. 
\end{theorem}
\begin{proof} \rm   Let $\alpha_{ij}=a_{ij}+A_{ij}$ for all $1\leq i, j\leq n$ and let $$\A=\begin{pmatrix}
	\alpha_{11} & \alpha_{12}& \cdots & \alpha_{1n}\\
	\vdots & \vdots& \ddots & \vdots \\
	\alpha_{n1} & \alpha_{n2} & \cdots & \alpha_{nn}
	\end{pmatrix}$$
	be the coordinate matrix. Suppose that $V$ is linearly dependent. By Theorem \ref{md12} there exists a linearly dependent set of vectors $a_i=\{a_{i1},...,a_{in}\} \in \R^n, $ where $a_i\in \alpha_i$ is a representative of $\alpha_i$ for all $i\in\{1,..,n\}$. It follows that 
	\begin{alignat*}{2}
	&\det(\A) = \det\begin{pmatrix}
	a_{11}+A_{11} & a_{12}+A_{12}& \cdots & a_{1n}+A_{1n}\\
	\vdots & \vdots& \ddots & \vdots \\
	a_{n1}+A_{n1} & a_{n2}+A_{n2} & \cdots & a_{nn}+A_{nn}
	\end{pmatrix}\\
	&= \det\begin{pmatrix}	a_{11} & a_{12}& \cdots & a_{1n}\\
	\vdots & \vdots& \ddots & \vdots \\
	a_{n1} & a_{n2} & \cdots & a_{nn}
	\end{pmatrix}+N\big(\det(\A)\big)
	= 0+N\big(\det(\A)\big)=N\big(\det(\A)\big),
	\end{alignat*}
	which is a contradiction. 	
\end{proof}	
The converse holds obviously for $n=1$. For $n=2$ the converse follows from the following proposition.
\begin{proposition} \label{md14new} A set  $V=\{\alpha_1,\alpha_2\}$ in $ \mathbb{E}^2$, where $\alpha_i=(\alpha_{i1},\alpha_{i2})$ for $1\leq i\leq 2$, is linearly independent if and only if  $\det\begin{pmatrix}
	\alpha_{11} & \alpha_{12}\\
	\alpha_{21} & \alpha_{22}
	\end{pmatrix}$  
	is zeroless. 
\end{proposition}
\begin{proof} 
	The sufficient condition is proved in Theorem \ref{md14}. 
	Assume  that the set of vectors $\{\alpha_1,\alpha_2\}$ is linearly independent. Suppose that  $	\det(\A)=\alpha_{11}\alpha_{22}-\alpha_{21}\alpha_{12}=N$ is a neutrix. Then it follows from \eqref{detrep} that  there exist representatives $ a_{ij}\in \alpha_{ij},1\leq i,j,\leq 2 $ such that $a_{11}a_{22}-a_{21}a_{12}=0$. This implies that the set of vectors $V=\{x_1, x_2\}$ with $x_1=(a_{11}, a_{12}), x_{2}=(a_{21}, a_{22})$ is  linearly dependent. By Theorem \ref{md12}, the set of vectors $\{\alpha_1,\alpha_2\}$ is linearly dependent, a contradiction.  		
\end{proof} 	

\begin{example}\rm    The set of vectors $$\{\eta_1=(1+\oslash,2+\epsilon\oslash), \eta_2=(-1+\epsilon \oslash,\epsilon\oslash)\}\subset\mathbb{E}^2$$ is linearly independent, since
	$$\det\begin{pmatrix}
	1+\oslash&2+\epsilon\oslash\\
	-1+\epsilon \oslash&\epsilon\oslash\\
	\end{pmatrix}=2+\epsilon\oslash.$$
	
\end{example}

We show now that the minor rank is always less than or equal to the row rank, and then study the relation with the strict rank.

\begin{theorem}\label{dl1.5} Let $\A=(\alpha_{ij})_{m\times n}\in \M_{m, n}(\E)$ with $\mr(\A)=r$. Then there exists a linearly independent set of $r$ row vectors of $\A$. As a consequence $r(\A)\geq \mr(\A)$.
\end{theorem}
\begin{proof} \rm Because $\mr(\A)=r$, we may suppose  without loss of generality that the minor $$M=\det \begin{pmatrix} \alpha_{11} & \cdots & \alpha_{1r} \\
	\vdots & \ddots & \vdots \\
	\alpha_{r1}& \cdots & \alpha_{rr}
	\end{pmatrix}  $$ 
	is zeroless. Let $\xi_i=(\alpha_{i1},\dots, \alpha_{in}),  1\leq i \leq m$ be row vectors of $\A$ and  $\xi'_i=(\alpha_{i1}, \dots, \alpha_{ir}), 1\leq i\leq m$ be vectors in $\E^r$. By Theorem \ref{md14} and the fact that $\det(M)$ is zeroless, the  set of vectors $\{\xi'_1,...,\xi'_r\}$ is linearly independent. 
	
	In order to prove that the  set of vectors $\{\xi_1,\dots, \xi_r\}$ is linearly independent, assume that $t_1 \xi_1+\cdots+ t_r \xi_r=(A_1, \dots, A_n)$, with $ A_1, \dots, A_n $ neutrices. Then $ t_1 \alpha_{1j}+t_2\alpha_{2j}+\cdots+t_r\alpha_{rj}=A_j$ for $ 1\leq j\leq n$. It follows that $t_1\xi_1'+\cdots +t_r\xi_r'=(A_1,\dots, A_r)$. Because $\{\xi_1', \dots, \xi_r'\}$ is linearly independent, it holds that $t_1=\cdots=t_r=0.$ Hence the set of vectors $\{\xi_1,...,\xi_r\} $  is linearly  independent by Proposition \ref{dauhieudoclaptuyentinh}.
\end{proof}

We show now that if the strict rank is defined, it is equal to the minor-rank and the row-rank.

\begin{theorem} \label{hanghaimatran}Let $\A$ be an $m\times n$ matrix over $\E$. If $\sr(\A)=r$, then $\mr(\A)=r(\A)=r$.
\end{theorem}
\begin{proof} First, because $\sr(\A)=r$, there exists a zeroless minor of order $r$ of $\A$. By the definition of minor-rank $\mr(\A)\geq r $. Let $\A_k=\A_{i_1\dots i_k, i_1\dots i_k}$  be a minor of order $k$ of $\A$ with $k>r$. Because there exists a representative matrix $\hat{\A}=(a_{ij})$ of $\A$ such that $\rank(\hat{\A})=r$, we have $\det\big(\hat{\A}_k\big)=\det\big(\hat{\A}_{i_1\dots i_k, i_1\dots i_k}\big)=0$. So $\det \big(A_{i_1\dots i_k, i_1\dots i_k}\big)$  is a neutrix. One concludes that  $\mr(\A)=r$.
	
	As or the second part, knowing that $\mr(\A)=r$, by Theorem \ref{dl1.5} there are at least $r$ linearly independent row vectors in $\A$. On the other hand there exists a representative matrix $\hat{\A}
	$ of $\A$ such that $\rank(\hat{\A})=r$.	 Without loss of generality, we may assume that $\det(\hat{\A}_r)=\det\begin{pmatrix}
	a_{11}& \cdots & a_{1r}\\
	\vdots & \cdots & \vdots\\
	a_{r1} & \cdots & a_{rr}
	\end{pmatrix} \not= 0.$ Let $i\in \{r+1, n\}$. Then the  set  of vectors 
	$$\{a_1=(a_{11}, \dots, a_{1n}), \dots, a_{r}=(a_{r1}, \dots, a_{rn}), a_{i}=(a_{i1}, \dots, a_{in})\}$$ is linearly dependent. By Theorem \ref{md12}.\ref{md12i} the set of vectors
	$$\{\alpha_1=(\alpha_{11},\dots, \alpha_{1n}), \dots, \alpha_{r}=(\alpha_{r1},\dots, \alpha_{rn}), \alpha_i=(\alpha_{i1},\dots, \alpha_{in})\}$$ is linearly dependent. So the row rank is at most $r$. 
	
	Combining we obtain that $r(\A)=r$.
\end{proof}

It follows from the next proposition that, if we define a column-rank by analogy to \ref{hangmatran}.\ref{row}, in the presence of the strict rank it is equal to the row-rank.

\begin{proposition}
	Let $\A=(\alpha_{ij})$ be an $m\times n$ matrix over $\E$. Then $\mr(\A)=\mr(A^T)$.
\end{proposition}
\begin{proof}
	It is a consequence of the fact that the determinant of a (sub)matrix is equal to the determinant of its transpose.
\end{proof}
We end this section by studying several conditions such that the strict rank is well-defined, implying that the minor-rank, the row-rank and the strict rank are equal.
\begin{theorem}\label{strick rank of matrix1}
	Let $\A=(\alpha_{ij})_{m\times n}$ be a matrix over $\E$. Assume that $r(\A)=r$  and there is a zeroless minor of order $r$ of $\A$. Then $\sr(\A)=r.$
\end{theorem}
\begin{proof}
	A linearly independent set of row vectors of $\A$ has up to $r$ elements, so by Theorem \ref{md12} the same is true for a set of representative vectors $V=\{a_1,\dots, a_m\}$, where $a_i\in \alpha_i=(\alpha_{i1}, \dots, \alpha_{in})$ for all $1\leq i\leq m$. It follows that the rank of the matrix $\hat{\A}=(a_{ij})$ is $r$. Also  there exists a zeroless minor of order $r$ of $\A$, hence $\sr(\A)=r.$
\end{proof}

Let $\A=(\alpha_{ij})_{m\times n}\equiv (a_{ij}+A_{ij})_{m\times n}\in \M_{m,n}(\E)$. It was observed in Section~\ref{section2c3} that only for $ m=n\leq 2 $ there is an obvious relation between the determinants given by Definition~\ref{defdet} and determinants of representatives. So direct conditions, without recurring to the strict rank, for the equality between row rank and minor rank possibly only can be given for matrices of low rank. For rank $ 1 $  Theorem~\ref{relationOfRankIndepent} considers the case that the minor rank is equal to the row rank, and then also equal to the strict rank, and Theorem~\ref{rowminor} the reverse case for rank $ 1 $ or $ 2 $. Observe that at least some element of $ \A $  must be zeroless, and for simplicity we assume that $\alpha_{11}$ is zeroless.

\begin{notation}\rm
	Let $\A=(\alpha_{ij})_{m\times n}\equiv (a_{ij}+A_{ij})_{m\times n}\in \M_{m,n}(\E)$ such that $\mr(\A)=1$. For $ 1\leq i\leq m $ we denote the $ i^{th} $ row vector by $ \alpha_{i} \equiv(\alpha_{i1},\cdots,\alpha_{in})$, and write $\overline{A}_{1}\equiv\displaystyle\max_{ 1\leq i\leq m}A_{i1},$ and  $\underline{A}^C_{1}=\displaystyle\min_{\substack{2\leq j\leq n\\ 1\leq i\leq m}} A_{ij}$. 
\end{notation}

\begin{theorem}
	\label{relationOfRankIndepent}
	Let $\A=(\alpha_{ij})$ be an $m\times n$ reduced matrix over $\E$, with $\alpha_{ij}=a_{ij}+A_{ij}$ for all $1\leq i\leq m$ and $1\leq j\leq n$.
	Assume that and $\mr(\A)=1$ and  $\alpha_{11}$ is zeroless. Suppose 
	that (i) $\dfrac{\oA_{1}}{\alpha_{11}}\subseteq \underline{A}^C_{1}$ for $1\leq i\leq m$, or (ii) all $A_{ij}$ are equal to some neutrix $ A $, where $1\leq i\leq m$ and $1\leq j\leq n$. Then $r(\A) =\sr(\A)=1.$
\end{theorem}
\begin{proof} The result is obvious for $m=1$. Assume that $1<m$. We will show that every set $\{\alpha_1,   \alpha_i\}$ is linearly dependent, where $i\in \{2, \dots, m\}$. In view of Theorem~\ref{strick rank of matrix1} we prove first that there exists a set of representative vectors $$\{a_{1}=(a_{11}, \dots, a_{1n}),  a_{1}=(a_{i1}, \dots, a_{in})\}$$ of  $\{\alpha_1,  \alpha_i\},$ such that the set of vectors $\{a_{1}, a_i\}$ is linearly dependent.	
	To do so, we prove that there is a set of vectors $$\{a_1=(a_{11}, \dots, a_{1n}),  a_{i}=(a_{i1},\dots, a_{in})\},$$ with  $a_{pq}\in \alpha_{pq}, p\in \{1,i\}, q\in \{1, \dots, n\}$  satisfying 
	\begin{equation}\label{equdtn}
	\det\begin{pmatrix}
	a_{11}& a_{1j}\\
	a_{i1} &  a_{ij}
	\end{pmatrix}=0,
	\end{equation}
	for all $j\in \{2, \dots, n\}$.
	
	For $j=2$, because $\mr(\A)=1$ the determinant $$\det\begin{pmatrix}
	\alpha_{11}& \alpha_{12}\\
	\alpha_{i1} & \alpha_{i2}
	\end{pmatrix}$$ is a neutrix. Consequently, there exists $a_{ps}\in \alpha_{ps}$ for all $p\in \{1, i\}, s\in \{1,2\} $ such that
	\begin{equation} \label{fixed element} \det\begin{pmatrix}
	a_{11}& a_{12}\\
	a_{i1} & a_{i2}
	\end{pmatrix}=0.
	\end{equation} 
	Hence formula  \eqref{equdtn}  is true for $j=2$. Let  $k\in \N, 2<k\leq n$ be arbitrary. We need to prove that there is a column $a_k=(a_{1k}, a_{ik})^T$ such that $a_{pk}\in \alpha_{pk}$ for $p\in \{1, i\}$ and 
	\begin{equation}\label{dieukiendoclaptuyetinh}  \det \begin{pmatrix}
	a_{11}&a_{1k}\\
	a_{i1}& a_{ik}
	\end{pmatrix}=0,
	\end{equation} 
	where $a_{11}, a_{i1}$ are defined by \eqref{fixed element}. Again because   $\mr(\A)=1$, the determinant 
	$$\det \begin{pmatrix}
	\alpha_{11}& \alpha_{1k}\\
	\alpha_{i1}  & \alpha_{ik}
	\end{pmatrix}$$ is a neutrix. As a result, there exists a matrix of representatives $\begin{pmatrix}
	a'_{11}& a'_{1k}\\
	a'_{i1} & a'_{ik}
	\end{pmatrix}$ with $a'_{ij}\in \alpha_{ij}$ such that \begin{equation}\label{equphu3n}
	\det \begin{pmatrix}
	a'_{11}& a'_{1k}\\
	a'_{i1} & a'_{ik}
	\end{pmatrix}=0.
	\end{equation}
	
	Case (i): We put 	
	$$d=a_{11},t= \det\begin{pmatrix}
	a_{11} & a'_{1k}\\
	a_{i1} & a'_{ik}
	\end{pmatrix}, $$ 
	and 
	\begin{equation*}\label{equphun} 
	\eps_{11}=a_{11}-a'_{11}, \eps_{i1}=a_{i1}-a'_{i1}, \eps_{ik}= -\dfrac{t}{d}.
	\end{equation*} 	
	
	Observe first that $\eps_{q1}\in A_{q1}$ for all $ q\in \{1,i\}$. We show that also $\eps_{ik}\in A_{ik} $. By \ref{equphu3n} one has 	
	\begin{alignat*}{2}
	t=&\det\begin{pmatrix}
	a'_{11}+\eps_{11} &  a'_{1k}\\
	a'_{i1}+\eps_{i1} &  a'_{ik}
	\end{pmatrix}\\
	=&\det\begin{pmatrix}
	a'_{11} &  a'_{1k}\\
	a'_{i1} &  a'_{ik}
	\end{pmatrix}+ \det\begin{pmatrix}
	\eps_{11} &  a'_{1k}\\
	\eps_{i1} &  a'_{ik}
	\end{pmatrix}=\eps_{11}a'_{ik}-\eps_{i1}a'_{1k}.
	\end{alignat*}
	Because 	$\eps_{p1}\in A_{p1}\subseteq \overline{A}_{1}$ for  $ p\in \{1, i\}$ and $|a'_{hk}|\leq |\alpha_{hk}|\leq 1+\oslash$ for $  h\in \{1, i\}$,  it follows that $t\in \overline{A}_{1}.$ Also  $d=a_{11}\in \alpha_{11}$. We conclude that
	\begin{equation*}   \label{equphu7n}
	\eps_{ik}= -\dfrac{t}{d}\in  \dfrac{\oA_{1}}{d}\subseteq \underline{A}^C_{1}\subseteq A_{ik}.
	\end{equation*} 
	Hence $a_{pk}\in \alpha_{pk}$ for all $p\in \{1, i\}$ with $a_{.k}=(a_{1k}, a_{ik})\equiv(a'_{1k}, a'_{ik}+\eps_{ik})^{T} $. In addition $a_{.k}$ satisfies formula \eqref{dieukiendoclaptuyetinh}, for
	\begin{align*}
	\det \begin{pmatrix}
	a_{11} & a'_{1k}\\
	a_{i1} & a'_{ik}+\eps_{ik}
	\end{pmatrix}&=t+ 	
	\det\begin{pmatrix}
	a_{11} &0\\
	a_{i1} &  \eps_{ik}
	\end{pmatrix}\\
	&=t+\eps_{ik}d=t-\dfrac{t}{d}d=0.
	\end{align*}
	
	Case (ii), Without loss of generality, we assume also that $|\alpha_{11}|$ is maximal. Put $u_1=( a_{11},a_{i1}	) $. The set of column vectors  
	$$\left\{ u'_1=(a'_{11},  a'_{i1})^{T}, 	u'_{k}=(a'_{1k}, a'_{ik})^T\right\} $$
	is linearly dependent. As a consequence, there exist real numbers $s$ and $ (\delta_{11},  \delta_{i1}) \in (A, A) $ such that 
	\begin{alignat}{2}\label{proofdlii} 
	u'_{k}=& su'_{1}	= s(u_1+\delta_1)	= su_1 + s\delta_1
	\end{alignat}
	where $\delta_1\equiv(\delta_{11},  \delta_{i1}) \in (A, A)$ and $s=a'_{1k}/a'_{11}$. Moreover $|s|\leq 1+\oslash$, since $|\alpha_{11}|$ is maximal. So $s\delta_1\in (A, A)$.

	Put $$u_{k}=u'_{k}-s\delta_1\equiv (a_{1k}, a_{ik})^T.$$ Then $a_{qk}\in \alpha_{qk}$ for $ q\in \{1, i\}$. By \eqref{proofdlii} one has $ u_{k}=su_1 $, so $\{u_1, u_{k}\}$ is linearly dependent. Hence
	$$\det\begin{pmatrix}
	a_{11}&  a_{1k}\\
	a_{i1} & a_{ik}
	\end{pmatrix}=0,$$ 
	which amounts again to \eqref{dieukiendoclaptuyetinh}.
	
	In both cases, because $ k $ is arbitrary, formula \eqref{equdtn} holds for $j=2, \dots, n$. 	
	We conclude that the set of vectors $\{a_1, a_p\}$ is linearly dependent. Then  $\{\alpha_1,\alpha_p\}$ is linearly dependent for all $p\in \{2, \dots, m\}$. So $r(\A)=1$ by Theorem \ref{md12}.	
	The last conclusion follows by Theorem \ref{strick rank of matrix1}. 
\end{proof}
%{\color{red} Note that conditions in (i) and (ii) are not the same. For example, the matrix \begin{equation*}
%\A=\begin{pmatrix}
%\eps+\eps\oslash &\eps\oslash\\
%\eps\oslash & \eps+\eps\oslash
%\end{pmatrix} 
%\end{equation*} with $\eps\simeq 0, \eps\not=0,$
%satisfies the condition in (ii), but not in (i). We have also that $r(\A)=2$ because $\mr(\A)=2.$ Of course, there are many matrices which satisfy condition in (i), but not in (ii).} 

%{\color{blue} The following example shows that the conditions are necessary, that is,  even for $\mr(\A)=1$, it may occur that the minor rank is less than te row-rank if a matrix is not square. (note that it holds for r=1 or r=2 without condition if the given matrix is square). 

%Consider the matrx $$\A=\begin{pmatrix}
%	1+\oslash & 1 &1+\eps\\
%	1& 1+\oslash & 1
%	\end{pmatrix}$$ with $\eps\simeq 0, \eps\not=0.$
%Then $\mr(\A)=1$ but $r(\A)=2.$ Indeed, we just need to verify that $r(\A)=2.$ Let $x_1=(1+\lambda, 1, 1+\eps)\in \alpha_{1}, x_2=(1, 1+\nu, 1)\in \alpha_{2}$ where $\lambda, \nu\in\oslash.$ We show that $\{x_1, x_2\}$ is linearly independent. 

%Suppose $\begin{vmatrix}
%1+\lambda &1\\
%1& 1+\nu
%\end{vmatrix}=0.$ It follows that $\nu=\lambda=0$. However, $\begin{vmatrix}
%1& 1+\eps\\
%1& 1
%\end{vmatrix}\not=0$. This means $\{x_1, x_2\}$ is linearly independent. it follows that $\{\alpha_1, \alpha_2\}$ is linearly independent. 

%} 
\begin{theorem}\label{rowminor} Let $\A=(\alpha_{ij})_{m\times n}\in \M_{m\times n}(\E)$. 
	Assume that $r(\A)=r\leq\min\{m, n\}$. If (i) $r=1$ or (ii) $r=2$ and all $A_{ij}$ are equal to some neutrix $A$, then  $\mr(\A)=r.$ As a result, $\sr(\A)=r.$
\end{theorem}
\begin{proof}\begin{enumerate}[(i)]
		
		\item Because the row rank of $\A$ is $1$, some  $\alpha_{pq}$ is zeroless, where $p\in \{1,\dots, m\}, q\in \{1,\dots, n\}$. This implies that $\mr(\A)\geq 1$. Also  $\mr(\A)\leq r(\A)=1$ by Theorem \ref{dl1.5}. It follows that $\mr(\A)=1.$ 
		
		\item By Theorem \ref{dl1.5} it holds that $\mr(\A)\leq r(\A)=2.$ Suppose that all minors of order $2$ are neutricial. Then $\mr(\A)\leq 1.$ If $\mr(\A)=0$, then $r=0$, a contradiction. If $\mr(\A)=1$, by part (ii) of Theorem \ref{relationOfRankIndepent} also $r(\A)=1$, again a contradiction.  Hence there exists a minor of order 2 which is zeroless. This means $\mr(\A)\geq 2$. Combining, we obtain that $\mr(\A)=2$.
	\end{enumerate}
\end{proof}
If it exists, the strict rank is more operational than the other notions, for example it permitted us to prove equality of the minor-rank and the row-rank, and then also the column-rank. In \cite{Nam2} it was helpful in solving singular systems of linear equations with coefficients and second member in terms of external numbers (the flexible systems of \cite{Jus}). By applying Gauss-Jordan elimination we may obtain a upper triangular matrix, with possibly some lines entirely composed by neutrices. If among others the strict rank is well-defined, Theorem 3.11 of \cite{Nam2} allows us to neglect these lines, leading to a closed form for the solution.

\end{document}